\documentclass[reqno]{amsart}

\usepackage{tikz-cd}

\usepackage{fullpage,dsfont}

\usepackage{hyperref} 

\usepackage{parskip}
\makeatletter 
\def\thm@space@setup{%
 \thm@preskip=\parskip \thm@postskip=0pt
}
\def\th@remark{%
  \thm@headfont{\itshape}%
  \normalfont 
  \thm@preskip\parskip \thm@postskip=0pt
}
\makeatother

\usepackage[nobysame,alphabetic,initials,msc-links]{amsrefs}

\usepackage[final]{pdfpages}

\usepackage{multirow}

\usepackage{stmaryrd}

\usepackage{array}

\DefineSimpleKey{bib}{how}
\DefineSimpleKey{bib}{mrclass}
\DefineSimpleKey{bib}{mrnumber}
\DefineSimpleKey{bib}{fjournal}
\DefineSimpleKey{bib}{mrreviewer}

\renewcommand{\PrintDOI}[1]{%
  \href{http://dx.doi.org/#1}{{\tt DOI:#1}}%
}
\renewcommand{\eprint}[1]{#1}
\BibSpec{book}{%
    +{}  {\PrintPrimary}                {transition}
    +{.} { \PrintDate}                  {date}
    +{.} { \textit}                     {title}
    +{.} { }                            {part}
    +{:} { \textit}                     {subtitle}
    +{,} { \PrintEdition}               {edition}
    +{}  { \PrintEditorsB}              {editor}
    +{,} { \PrintTranslatorsC}          {translator}
    +{,} { \PrintContributions}         {contribution}
    +{,} { }                            {series}
    +{,} { \voltext}                    {volume}
    +{,} { }                            {publisher}
    +{,} { }                            {organization}
    +{,} { }                            {address}
    +{,} { }                            {status}
    +{,} { \PrintDOI}                   {doi}
    +{,} { \PrintISBNs}                 {isbn}
    +{}  { \parenthesize}               {language}
    +{}  { \PrintTranslation}           {translation}
    +{;} { \PrintReprint}               {reprint}
    +{.} { }                            {note}
    +{.} {}                             {transition}
    +{}  {\SentenceSpace \PrintReviews} {review}
}
\BibSpec{article}{%
    +{}  {\PrintAuthors}                {author}
    +{,} { \textit}                     {title}
    +{.} { }                            {part}
    +{:} { \textit}                     {subtitle}
    +{,} { \PrintContributions}         {contribution}
    +{.} { \PrintPartials}              {partial}
    +{,} { }                            {journal}
    +{}  { \textbf}                     {volume}
    +{}  { \PrintDatePV}                {date}
    +{,} { \issuetext}                  {number}
    +{,} { \eprintpages}                {pages}
    +{,} { }                            {status}
    +{,} { \PrintDOI}                   {doi}
    +{,} { \eprint}        {eprint}
    +{}  { \parenthesize}               {language}
    +{}  { \PrintTranslation}           {translation}
    +{;} { \PrintReprint}               {reprint}
    +{.} { }                            {note}
    +{.} {}                             {transition}
    +{}  {\SentenceSpace \PrintReviews} {review}
}
\BibSpec{collection.article}{%
    +{}  {\PrintAuthors}                {author}
    +{,} { \textit}                     {title}
    +{.} { }                            {part}
    +{:} { \textit}                     {subtitle}
    +{,} { \PrintContributions}         {contribution}
    +{,} { \PrintConference}            {conference}
    +{}  {\PrintBook}                   {book}
    +{,} { }                            {booktitle}
    +{,} { \PrintDateB}                 {date}
    +{,} { pp.~}                        {pages}
    +{,} { }                            {publisher}
    +{,} { }                            {organization}
    +{,} { }                            {address}
    +{,} { }                            {status}
    +{,} { \PrintDOI}                   {doi}
    +{,} { \eprint}        {eprint}
    +{}  { \parenthesize}               {language}
    +{}  { \PrintTranslation}           {translation}
    +{;} { \PrintReprint}               {reprint}
    +{.} { }                            {note}
    +{.} {}                             {transition}
    +{}  {\SentenceSpace \PrintReviews} {review}
}
\BibSpec{misc}{%
  +{}{\PrintAuthors}  {author}
  +{,}{ \textit}      {title}
  +{.}{ }             {how}
  +{}{ \parenthesize} {date}
  +{,} { available at \eprint}        {eprint}
  +{,}{ available at \url}{url}
  +{,}{ }             {note}
  +{.}{}              {transition}
}
\usepackage{amssymb, amsfonts, amsxtra, amsmath}
\usepackage{mathrsfs}
\usepackage{mathdots}
\usepackage{wasysym}
\usepackage[all]{xy}
\usepackage{bbm}
\usepackage{calc}
\usepackage{accents}

\usepackage[bbgreekl]{mathbbol}			

\usepackage{stackengine}

\numberwithin{equation}{section}

\DeclareSymbolFontAlphabet{\mathbb}{AMSb}	
\DeclareSymbolFontAlphabet{\mathbbl}{bbold}	

\newtheorem{Theorem}{Theorem}[section]
\newtheorem*{Theorem*}{Theorem}
\newtheorem{Def}[Theorem]{Definition}
\newtheorem*{Def*}{Def}
\newtheorem{Lem}[Theorem]{Lemma}
\newtheorem{Prop}[Theorem]{Proposition}
\newtheorem{Cor}[Theorem]{Corollary}

\newtheorem{Rem}[Theorem]{Remark}

\newtheorem{Exa}[Theorem]{Example}

\newcommand\bp{\begin{proof}}
\newcommand\ep{\end{proof}}

\mathchardef\mhyph="2D

\DeclareMathOperator{\id}{\mathrm{id}}

\DeclareMathOperator{\Mor}{\mathrm{Mor}}

\DeclareMathOperator{\Tr}{\mathrm{Tr}}

\DeclareMathOperator{\Irr}{\mathrm{Irr}}

\DeclareMathOperator{\alg}{\mathrm{alg}}

\newcommand{\op}{\mathrm{op}}

\newcommand{\msN}{\mathscr{N}}

\newcommand{\G}{\mathbb{G}}

\newcommand{\ovot}{\bar{\otimes}}

\begin{document}

\title{Equivariant injectivity of crossed products}
\author{Joeri De Ro}
\address{Vrije Universiteit Brussel}
\email{joeri.ludo.de.ro@vub.be}
\maketitle

\begin{abstract} We introduce the notion of a $\G$-operator space $(X, \alpha)$, which consists of an action $\alpha: X \curvearrowleft \G$ of a locally compact quantum group $\G$ on an operator space $X$, and we make a study of the notion of $\G$-equivariant injectivity for such an operator space. Given a $\G$-operator space $(X, \alpha)$, we define a natural associated crossed product operator space $X\rtimes_\alpha \G$, which has canonical actions $X\rtimes_\alpha \G \curvearrowleft \G$ (the adjoint action) and $X\rtimes_\alpha \G\curvearrowleft \check{\G}$ (the dual action) where $\check{\G}$ is the dual quantum group. We then show that if $X$ is a $\G$-operator system, then $X\rtimes_\alpha \G$ is $\G$-injective if and only if $X\rtimes_\alpha \G$ is injective and $\G$ is amenable, and that (under a mild assumption) $X\rtimes_\alpha \G$ is $\check{\G}$-injective if and only if $X$ is $\G$-injective. We discuss how these results generalise and unify several recent results from the literature, and we give new applications.
\end{abstract}

\section{Introduction}

In \cite{Ham91}, Hamana introduced the notion of an $M$-comodule, where $(M, \Delta)$ is a Hopf-von Neumann algebra, as an operator space $X$ together with a completely isometric map $\alpha: X \to X \ovot M$ satisfying the coaction property $(\alpha\otimes \id)\alpha = (\id \otimes \Delta)\alpha$, where $\ovot$ denotes the Fubini tensor product of operator spaces. Two main examples are of particular interest, namely those arising from a locally compact group $G$: 
\begin{itemize}
    \item $M= L^\infty(G)$ with $\Delta(f)(s,t)= f(st)$ for $f\in L^\infty(G)$ and $s,t \in G$.
    \item $M= \mathscr{R}(G)$, the right group von Neumann algebra, with $\check{\Delta}(\rho_g) = \rho_g\otimes \rho_g$ for $g\in G$ and where $\rho: G \to U(B(L^2(G)))$ is the right regular representation of $G$.
\end{itemize}

If $\G$ is a locally compact quantum group with associated Hopf von Neumann algebra $(L^\infty(\G), \Delta)$ (see Section 2 for notations and conventions), it is then natural to define a $\G$-operator space $X$ as an $L^\infty(\G)$-comodule. In particular, keeping in mind that $(\mathscr{R}(G), \check{\Delta})$ can be seen as the function algebra of the dual locally compact quantum group $\check{G}$, the two examples of interest mentioned above fall into this category. Given a $\G$-operator space $(X, \alpha)$, we will write $X^\alpha = \{x\in X: \alpha(x)= x \otimes 1\}$ for its space of $\G$-fixed points. A $\G$-operator system $(X, \alpha)$ is a $\G$-operator space $(X, \alpha)$ such that $X$ is an operator system and $\alpha$ is unital.

We call a $\G$-operator space $X$ $\G$-injective if for every two $\G$-operator spaces $Y,Z$, every $\G$-equivariant complete contraction $\phi: Y \to X$ and every $\G$-equivariant complete isometry $\iota: Y \to Z$, there is a $\G$-equivariant complete contraction $\Phi: Z \to X$ such that $\Phi\circ \iota = \phi$. 

The main goal of this paper is to study the notion of $\G$-injectivity. In particular, we prove the following result that relates different notions of $\G$-injectivity:

\textbf{Proposition \ref{injective}.}\textit{
    Let $(X, \alpha)$ be a $\G$-operator space. The following statements are equivalent:
    \begin{enumerate}
    \item $X$ is injective as an $L^1(\G)$-module for the action $\omega \rhd x:= (\id \otimes \omega)\alpha(x)$ with $\omega \in L^1(\G)$ and $x\in X$.
        \item $(X, \alpha)$ is $\G$-injective.
        \item $X$ is injective (as an operator space) and $(X, \alpha)$ is $\G$-amenable.
    \end{enumerate}}

Here, we call a $\G$-operator space $(X, \alpha)$ $\G$-amenable if there exists a $\G$-equivariant completely contractive map $E: (X \ovot L^\infty(\G), \id \otimes \Delta)\to (X, \alpha)$ such that $E\circ \alpha = \id_X$ (following the terminology of \cite[Definition 6.6]{DD24} for $\G$-dynamical von Neumann algebras). The map $E$ should be thought of as a $\G$-averaging procedure for the action $\alpha$. The result also clarifies the connection with the notion of $L^1(\G)$-module injectivity as discussed in e.g.\ \cite[Section 2]{Cr17a}. In fact, the notion of $\G$-equivariant injectivity (in one form or another) has been studied by many people, see e.g.\ \cite{AD79, AD82, Ham82, Ham85, Ham91, BEW20, MP22} for classical groups and \cite{CN16, Cr17a, Cr17b, Moa18, Cr21, BC21, KKSV22, HHN22, HY22, DH24, AK24} for quantum groups. 

If $(X, \alpha)$ is a $\G$-operator space, we generalise Hamana's definition of the Fubini crossed product \cite[Definition 5.4]{Ham91} and we define the associated crossed product operator space $X\rtimes_\alpha \G$. We have canonical actions $\id \otimes \Delta_r: X\rtimes_\alpha \G \curvearrowleft \G$ and $\id \otimes \check{\Delta}_r: X\rtimes_\alpha \G \curvearrowleft \check{\G}$, so it makes sense to consider the iterated crossed products
$$(X\rtimes_\alpha \G)\rtimes_{\id \otimes \Delta_r} \G, \quad (X\rtimes_\alpha \G)\rtimes_{\id \otimes \check{\Delta}_r}\check{\G}.$$
We show that we always have $(X\rtimes_\alpha \G)\rtimes_{\id \otimes \Delta_r} \G \cong (X\rtimes_\alpha \G)\ovot L^\infty(\check{\G})$ and that under a technical assumption on the $\G$-operator space $X$, we have the Takesaki-Takai identification $(X\rtimes_\alpha\G)\rtimes_{\id \otimes \check{\Delta}_r} \check{\G} \cong X \ovot B(L^2(\G))$ (cfr. \cite[Proposition 5.7]{Ham91}).

We will be especially interested in how the notion of equivariant injectivity behaves with respect to the crossed product construction. We will prove the following general result:

\textbf{Theorem \ref{main result} + Theorem \ref{mainresult2}.}\textit{
   Let $(X, \alpha)$ be a $\G$-operator system. Then
    \begin{enumerate}
        \item $X$ is $\G$-injective if and only if $X\rtimes_\alpha \G$ is $\check{\G}$-injective and $\alpha(X)= (X\rtimes_\alpha \G)^{\id \otimes \check{\Delta}_r}$.
        \item $X\rtimes_\alpha \G$ is $\G$-injective if and only if $X\rtimes_\alpha \G$ is injective (as an operator system) and $\G$ is amenable.
    \end{enumerate}}

The condition $\alpha(X)= (X\rtimes_\alpha \G)^{\id \otimes \check{\Delta}_r}$ is mild, and is for example automatically satisfied for $\G$-dynamical von Neumann algebras.

Similar results and its consequences have appeared in many forms or particular cases in the literature:

\begin{itemize}
    \item In \cite[Theorem 4.3]{DH24}, $(1)$ is proven in the case where $\G$ is a discrete quantum group. In fact, this paper can be seen as a continuation on the work done in \cite{DH24}.
    \item In \cite[Theorem 5.2]{BC21}, $(1)$ was proven in the case that $X$ is an injective von Neumann algebra, $\alpha: X \to X \ovot L^\infty(\G)$ a unital, normal, isometric $*$-morphism and where $\G$ is a locally compact group. In \cite[Theorem 3.5]{MP22}, the injectivity assumption on $X$ in \cite[Theorem 5.2]{BC21} was shown to be unnecessary.
    \item Applying $(1)$ with $X= \mathbb{C}$, we find back the statement that $\G$ is an amenable locally compact quantum group if and only if $L^\infty(\check{\G})$ is $\check{\G}$-injective, which is \cite[Theorem 5.2]{Cr17a}.
    \item Applying $(2)$ with $X= L^\infty(\G)$ and $\alpha = \Delta$, we find that $B(L^2(\G))$ is $\G$-injective if and only if $\G$ is amenable. In particular, if $\G$ is compact, then $B(L^2(\G))$ is $\G$-injective, which has appeared in \cite[Lemma 2.10]{HHN22} and \cite[Lemma 3.12]{DH24}. Compare also with \cite[Theorem 5.5]{CN16}.
    \item Injectivity of crossed products with respect to locally compact groups has been investigated in the von Neumann algebra setting in \cite{AD79, AD82} and in the operator space setting in \cite[Proposition 5.8]{Ham91}. 
\end{itemize}

From Proposition \ref{injective}, it is clear that a good understanding of the notion of an amenable action is crucial for the study of equivariant injectivity. Therefore, we spend some effort investigating (inner) amenable actions. For instance, as an application of general duality results for crossed products (see \cite[Theorem 6.12]{DD24}, Theorem \ref{main result}, Theorem \ref{mainresult2} and Corollary \ref{innercor}), we give a clean and conceptual proof of the following dynamical characterisation of (inner) amenability of a locally compact quantum group:

\textbf{Theorem \ref{amenable}.}\textit{ The following statements are equivalent:
\begin{itemize}
    \item[(a)] $\G$ is inner amenable, i.e. the trivial action $\tau: \mathbb{C}\curvearrowleft\G$ is inner amenable.
    \item[(b)] The action $\check{\Delta}: L^\infty(\check{\G})\curvearrowleft \check{\G}$ is amenable.
    \item[(c)] The action $\Delta_r: L^\infty(\check{\G})\curvearrowleft \G$ is inner amenable.
    \item[(d)] The action $\Delta_r: B(L^2(\G))\curvearrowleft \G$ is inner amenable.
\end{itemize} Moreover, all the following statements are equivalent:
    \begin{enumerate}
        \item $\G$ is amenable, i.e. the trivial action $\mathbb{C}\curvearrowleft \mathbb{G}$ is amenable.
          \item The action $\Delta_r: L^\infty(\check{\G})\curvearrowleft \G$ is amenable.
        \item The action $\Delta_r: B(L^2(\G))\curvearrowleft \G$ is amenable.
        \item The action $\check{\Delta}: L^\infty(\check{\G})\curvearrowleft \check{\G}$ is inner amenable.
        \item $(\mathbb{C}, \tau)$ is $\G$-injective.
         \item $(L^\infty(\check{\G}), \check{\Delta})$ is $\check{\G}$-injective.
         \item $(B(L^2(\G)), \Delta_r)$ is $\G$-injective.
        \item $L^\infty(\check{\G})$ is injective and $\G$ is inner amenable.
    \end{enumerate}}

Here, we call a $\G$-dynamical von Neumann algebra $(M, \alpha)$ inner amenable if there exists a $\G$-equivariant unital completely positive map $E: (M\rtimes_\alpha \G, \id \otimes \Delta_r)\to (M, \alpha)$ such that $E\circ \alpha = \id_M$ \cite[Definition 6.6]{DD24}. Some implications in this result are already known, e.g. $(1)\iff (6)$ is \cite[Theorem 5.2]{Cr17a} and $(1)\iff (8)$ is \cite[Corollary 3.9]{Cr19}, but our proofs follow a different strategy.

We also give an application of our results to the non-commutative Poisson boundary \cite{Izu02, KNR13}:

\textbf{Proposition \ref{Poisson}.}\textit{
If $\mu\in C_u(\G)^*$ is a state, consider the associated non-commutative Poisson boundary $\mathcal{H}_\mu$. 
    \begin{enumerate}
        \item[(a)] If $\check{\G}$ is inner amenable, then the action $\Delta_\mu: \mathcal{H}_\mu\curvearrowleft\G$ is amenable.
        \item[(b)] If $\check{\G}$ is amenable, then $\Delta_\mu: \mathcal{H}_\mu\curvearrowleft\G$ is $\G$-injective. Consequently, $\mathcal{H}_\mu\rtimes_{\Delta_\mu} \G$ is $\check{\G}$-injective.
    \end{enumerate}}

If $\G$ is a discrete quantum group, this follows from the results in \cite{Moa18} or \cite{DH24}. 

We also prove compatibility of equivariant injective envelopes (of operator systems) with crossed products for discrete quantum groups and compact quantum groups:

\textbf{Proposition \ref{injective envelope crossed product} + Proposition \ref{injective envelope crossed product2}.}\textit{ Let $\G$ be a discrete or a compact quantum group and $(X, \alpha)$ a $\G$-operator system. Then 
    $$I_{\check{\G}}^1(X\rtimes_\alpha \G)= I_\G^1(X)\rtimes \G$$
    as $\check{\G}$-operator systems.}

  Using the terminology of \cite{DH24}, Proposition \ref{injective envelope crossed product} can be seen as the analoguous result of \cite[Theorem 5.5]{DH24} but for $\G$-$W^*$-operator systems instead of $\G$-$C^*$-operator systems.

We also prove the following permanence properties:

   \textbf{Theorem \ref{main application} + Proposition \ref{subgroups}.}\textit{ Let $(M, \alpha)$ be a $\G$-dynamical von Neumann algebra, $\mathbb{H}$ a Vaes-closed quantum subgroup of $\G$, and $\alpha_{\mathbb{H}}: M \curvearrowleft \mathbb{H}$ the restriction of the action. Then we have:
    \begin{itemize}
        \item If $(M, \alpha)$ is $\G$-injective, then $(M, \alpha_{\mathbb{H}})$ is $\mathbb{H}$-injective.
        \item If $(M, \alpha)$ is inner $\G$-amenable, then $(M, \alpha_{\mathbb{H}})$ is inner $\mathbb{H}$-amenable.
    \end{itemize}}  
  \section{Preliminaries}

 All vector spaces and algebras are defined over the field $\mathbb{C}$. If $S$ is a subset of some normed algebra $A$, we write $[S]= \overline{\operatorname{span}(S)}^{\|\cdot\|}$. 

\subsection{Operator spaces and operator systems}

For the theory of operator spaces and operator systems, we refer the reader to the standard reference \cite{ER00}. All operator spaces involved are assumed to be complete. Given an inclusion $X\subseteq Y$ of operator spaces, a completely contractive map $\phi: Y \to X$ is called \emph{conditional expectation} if $\phi(x)= x$ for all $x\in X$.
Several kinds of tensor products will play an important role in this paper. We employ the following notations:

\begin{itemize}
    \item $\odot$ will denote the algebraic tensor product of vector spaces.
        \item $\hat{\otimes}$ will denote the \emph{projective tensor product} of operator spaces \cite[Chapter 7]{ER00}.
    \item $\otimes$ will denote the \emph{injective (or minimal) tensor product} of operator spaces \cite[Chapter 8]{ER00}.
    \item $\ovot$ will denote the \emph{Fubini tensor product} of operator spaces. 
\end{itemize}

The Fubini tensor product is not as well-known as the other operator space tensor products. Therefore, we give a brief overview of the relevant facts that we will be needing. The following discussion is taken from \cite[Section 2.1]{DH24}, with the difference that we will be needing operator spaces and completely bounded maps instead of operator systems and unital completely positive maps. Therefore, the discussion was slightly modified accordingly. For proofs of the facts that follow, we refer the reader to \cite[I, Section 3]{Ham82}. We note that Hamana's proofs and statements are given in the context of operator systems, but the proofs are easily adapted to the more general operator space setting. See also \cite[Section 1]{Ham91} for more information.

Let $X\subseteq B(\mathcal{H}),Y\subseteq B(\mathcal{K})$ be concretely represented operator spaces. We define the Fubini tensor product
$$X \ovot Y = \{z \in B(\mathcal{H}\otimes \mathcal{K}): (\mu \ovot \id)(z) \in Y \text{\ and }(\id \ovot \nu)(z)\in X \text{\ for\ all } \mu \in B(\mathcal{H})_*, \nu\in B(\mathcal{K})_*\}$$
which is again an operator space. The Fubini tensor product obeys obvious commutativity, associativity and distributivity (with respect to $\ell^\infty$-direct sums of operator spaces) properties. If either $X$ or $Y$ is finite-dimensional, then $X\ovot Y$ coincides with the algebraic tensor product $X\odot Y$. If $X,Y$ are both von Neumann algebras, then the Fubini tensor product is simply the von Neumann tensor product.

 \begin{Prop}\cite[I, Lemma 3.5]{Ham82}
     Assume that $X_j\subseteq B(\mathcal{H}_j), j=1,2$ are operator spaces and $\phi: X_1\to X_2$ is a completely bounded map. Fix an orthonormal basis $\{e_i\}_{i\in I}$ for $\mathcal{K}$ and let $\{E_{ij}\}_{i,j\in I}\subseteq B(\mathcal{K})$ be the associated matrix units. The assignment
     $$\phi\otimes \id: X_1\ovot B(\mathcal{K})\to X_2\ovot B(\mathcal{K}): \sum_{(i,j)} x_{ij}\otimes E_{ij}\mapsto \sum_{(i,j)} \phi(x_{ij})\otimes E_{ij}$$
     (with the sums converging in the strong topology) is a well-defined completely bounded map with $\|\phi\otimes \id\|_{\operatorname{cb}}= \|\phi\|_{\operatorname{cb}}.$ Moreover, it satisfies the following properties:
     \begin{enumerate}
         \item $\phi\otimes \id$ extends $\phi\odot \id: X_1\odot B(\mathcal{K})\to X_2\odot B(\mathcal{K})$.
         \item $(\phi\otimes \id)((1\otimes a)z(1\otimes b)) = (1\otimes a)(\phi\otimes \id)(z)(1\otimes b)$ for all $a,b\in B(\mathcal{K})$ and all $z\in X_1 \ovot B(\mathcal{K})$.
         \item $(\id \otimes \omega)(\phi\otimes \id)(z) = \phi((\id \otimes \omega)(z))$ for all $z\in X_1\ovot B(\mathcal{K})$ and all $\omega \in B(\mathcal{K})_*$.
     \end{enumerate}
    The properties $(1)$ and $(2)$ determine $\phi\otimes \id$ uniquely in the sense that if $\psi: X_1 \ovot B(\mathcal{K})\to X_2\ovot B(\mathcal{K})$ is a linear map satisfying $(1)$ and $(2)$, then $\psi= \phi\otimes \id$.
 \end{Prop}
 
Of course, we have the analoguous result for $\id \otimes \phi: B(\mathcal{K})\ovot X_1\to B(\mathcal{K})\ovot X_2$ as well. If $X\subseteq B(\mathcal{H})$ is an operator space and if $f\in X^*$, then we can form the slice map $f\otimes \id: X \ovot B(\mathcal{K})\to \mathbb{C}\ovot B(\mathcal{K})\cong B(\mathcal{K})$ and similarly for $\id \otimes f$.

To make sure that everything works well, we need to impose continuity assumptions. This is illustrated by the following results, which will be used all the time without further mention.

\begin{Lem}\label{Fubinitensor}\cite[I, Lemma 3.8]{Ham82}
   Let $X_j\subseteq B(\mathcal{H}_j), j=1,2$ be operator spaces, $Y\subseteq B(\mathcal{K})$ an operator space and $\phi: X_1\to X_2$ a completely bounded map. We have $(\phi \otimes \id)(X_1 \ovot Y)\subseteq X_2 \ovot Y$ in the following two cases:
    \begin{enumerate}
        \item $Y$ is $\sigma$-weakly closed.
        \item  $X_1$ and $X_2$ are $\sigma$-weakly closed and $\phi: X_1\to X_2$ is $\sigma$-weakly continuous
    \end{enumerate} A similar result holds for $\id \otimes \phi$.
\end{Lem}
\begin{Prop}\cite[I, Lemma 3.9]{Ham82}
   Let $X_j\subseteq B(\mathcal{H}_j)$ and $Y_j\subseteq B(\mathcal{K}_j)$, $j=1,2$ be operator spaces. Let $\phi: X_1\to X_2$ and $\psi: Y_1\to Y_2$ be completely bounded maps. If $Y_1$ and $Y_2$ are $\sigma$-weakly closed and $\psi: Y_1\to Y_2$ is $\sigma$-weakly continuous, then for $z\in X_1\ovot Y_1$, the elements
$$(\id_{B(\mathcal{H}_2)}\otimes\psi)(\phi \otimes\id_{B(\mathcal{K}_1)})(z), \quad (\phi \otimes \id_{B(\mathcal{K}_2)}) (\id_{B(\mathcal{H}_1)}\otimes \psi)(z)$$
agree and belong to $X_2\ovot Y_2$.
In this way, we obtain a canonical map $\phi\otimes \psi: X_1\ovot Y_1\to X_2\ovot Y_2$ with $\|\phi\otimes \psi\|_{\operatorname{cb}}= \|\phi\|_{\operatorname{cb}}\|\psi\|_{\operatorname{cb}}$. If $\phi, \psi$ are complete isometries, then so is $\phi\otimes \psi: X_1\ovot Y_1\to X_2\ovot Y_2$.
\end{Prop}

Given a von Neumann algebra $M$ and an operator space $X$, we can use the preceding result to make sense of the abstract operator space $X\ovot M$ without specifying concrete realisations as bounded operators on certain Hilbert spaces.

Let us also mention the following result:
\begin{Prop}\cite[I. Proposition 3.10]{Ham82} If $X\subseteq B(\mathcal{H})$ is an operator space and $Y\subseteq B(\mathcal{K})$ is a $\sigma$-weakly closed operator space, then $X \ovot Y$ is injective if and only if both $X$ and $Y$ are injective.
\end{Prop}

If one encounters a Fubini tensor product (possibly with more than two factors) in this paper, it will always be the case that all factors, except possibly the first one, are von Neumann algebras and if tensor product maps between Fubini tensor products are considered, then there will always be $\sigma$-weakly continuous maps acting on the factors of the tensor product that are von Neumann algebras. Hence, we do not have to worry too much about technicalities that come with the Fubini tensor product.

    If $X$ is an operator space and $M$ is a von Neumann algebra, then the map
    $$\Phi:X\ovot M\to CB(M_*, X), \quad \Phi(z)(\omega)= (\id \otimes \omega)(z), \quad z \in X \ovot M, \omega \in M_*$$
    is a completely isometric isomorphism \cite{Daw22}. To see this, note that if $X \subseteq B(\mathcal{H})$ for a Hilbert space $\mathcal{H}$, we have
    $$B(\mathcal{H})\ovot M\cong (B(\mathcal{H})_*\hat{\otimes} M_*)^* \cong CB(M_*, (B(\mathcal{H})_*)^*)\cong CB(M_*, B(\mathcal{H}))$$
    where we used \cite[Corollary 7.1.5, Theorem 7.2.4]{ER00}. It is then straightforward to verify that the isomorphism $B(\mathcal{H})\ovot M\cong CB(M_*, B(\mathcal{H}))$ restricts to the desired isomorphism $X\ovot M \cong CB(M_*, X)$. 

    We also need some language from the theory of operator modules \cite{Cr17a}. A \emph{completely contractive Banach algebra} $A$ is an algebra which also has an operator space structure for which the multiplication extends to a complete contraction $A\hat{\otimes} A \to A$. In this paper, we will mostly be interested in the case $A= L^1(\G)$ for a locally compact quantum group $\G$, with the algebra structure given by convolution of normal functionals (see the next subsection). Given a completely contractive algebra $A$, also the unitisation $A_1:= A \oplus_1 \mathbb{C}$ with its usual algebra structure is a completely contractive Banach algebra. Elements $(a, \lambda)$ in $A_1$ will be written as $a+ \lambda e$ where $e$ is to be thought of as the unit of the algebra. A left \emph{operator $A$-module} is a pair $(X, \rhd)$ where $X$ is an operator space and $\rhd: A\hat{\otimes} X\to X$ is a complete contraction such that $X$ becomes an algebraic left $A$-module for $\rhd$. If $X$ is an operator space, then the operator spaces $CB(A, X)$ and $CB(A_1, X)$ are left operator $A$-modules in the obvious way. If $X$ is a (left) operator $A$-module, there are canonical maps
    \begin{align*}
        &j_X: X \to CB(A, X), \ \quad j_X(x)(a) = a\rhd x,\\
        &j_X^1: X\to CB(A_1, X), \quad j_X^1(x)(a+ \lambda e) = a\rhd x + \lambda x
    \end{align*}
and then $j_X$ is an $A$-linear complete contraction and $j_X^1$ is an $A$-linear complete isometry. An operator $A$-module $X$ is is called \emph{$A$-injective} if for every two operator $A$-modules $Y,Z$, every $A$-linear complete isometry $\iota: Y \to Z$ and every $A$-linear complete contraction $\phi: Y \to X$, there is an $A$-linear complete contraction $\Phi: Z \to X$ such that $\Phi\circ \iota = \phi$. 
\subsection{Locally compact quantum groups.}

We recall some basic results from the theory of locally compact quantum groups \cite{KV00,KV03,VV03}. We will follow the same conventions and notations as e.g.\ \cite{DD24}. Note that the conventions for quantum group duality differ from our previous work \cite{DH24}.

A \emph{Hopf-von Neumann algebra} is a pair $(M, \Delta)$ where $M$ is a von Neumann algebra and $\Delta: M \to M \ovot M$ a unital, normal, isometric $*$-homomorphism such that $(\Delta \otimes \id)\circ \Delta = (\id \otimes \Delta)\circ \Delta$.

A \emph{locally compact quantum group} $\G$ is a Hopf-von Neumann algebra $(L^\infty(\G), \Delta)$ for which there exist normal, semifinite, faithful weights $\Phi, \Psi: L^\infty(\G)_+\to [0, \infty]$ such that $(\id\otimes \Phi)\Delta(x)= \Phi(x)1$ for all $x\in \mathscr{M}_\Phi^+$ and $(\Psi\otimes \id)\Delta(x)= \Psi(x)1$ for all $x\in \mathscr{M}_\Psi^+$. The von Neumann algebra predual $L^1(\G):= L^\infty(\G)_*$ is a completely contractive Banach algebra for the multiplication 
$$\mu\star \nu:= (\mu\otimes \nu)\circ \Delta, \quad \mu, \nu \in L^1(\G).$$

Without any loss of generality, we may (and will) assume that $L^\infty(\G)$ is faithfully represented in standard form on the Hilbert space $L^2(\G)$. With respect to the weights $\Phi, \Psi$, we then have GNS-maps
$$\Lambda_\Phi: \mathscr{N}_\Phi\to L^2(\G), \quad \Lambda_\Psi: \mathscr{N}_\Psi\to L^2(\G).$$

Fundamental to the theory of locally compact quantum groups are the unitaries 
\[
V,W \in B(L^2(\G)\otimes L^2(\G))
\]
called respectively \emph{right} and \emph{left} regular unitary representation. They are uniquely characterised by the identities 
\[
(\id\otimes \omega)(V) \Lambda_{\Psi}(x) = \Lambda_{\Psi}((\id\otimes \omega)\Delta(x)),
\qquad \omega \in L^1(\G),x\in \msN_{\Psi},
\]
\[
 (\omega \otimes \id)(W^*)\Lambda_{\Phi}(x) = \Lambda_{\Phi}((\omega\otimes \id)\Delta(x)),\qquad \omega \in L^1(\G),x\in \msN_{\Phi}.
\]
They are \emph{multiplicative unitaries} \cite{BS93} meaning that
\[
V_{12}V_{13}V_{23} = V_{23}V_{12},\qquad W_{12}W_{13}W_{23}= W_{23}W_{12}
\]
and they implement the coproduct of $L^\infty(\G)$ in the sense that
$$\label{EqComultImpl}
W^*(1\otimes x)W = \Delta(x) = V(x\otimes 1)V^*,\qquad x\in L^\infty(\G).$$
Moreover, we have
$$C_0(\G):=[(\omega\otimes \id)(V) \mid \omega \in B(L^2(\G))_*] = [(\id\otimes \omega)(W) \mid \omega \in B(L^2(\G))_*]$$
which is a $\sigma$-weakly dense C$^*$-subalgebra of $L^\infty(\G)$. Then $\Delta(C_0(\G))\subseteq M(C_0(\G)\otimes C_0(\G))$. 

We can also define the von Neumann algebra
$L^\infty(\hat{\G}) = [(\omega\otimes \id)(W) \mid \omega \in L^1(\G)]^{\sigma\textrm{-weak}}$
with coproduct
$\hat{\Delta}(x) = \Sigma W(x\otimes 1)W^*\Sigma$ where $x\in L^\infty(\hat{\G})$. 
Then the pair $(L^\infty(\hat{\G}), \hat{\Delta})$ defines the locally compact quantum group $\hat{\G}$, which is called the dual of $\G$. Similarly, we can also define the von Neumann algebra $L^\infty(\check{\G}) = [(\id\otimes \omega)(V) \mid \omega \in L^1(\G)]^{\sigma\textrm{-weak}}$
with coproduct 
$\check{\Delta}(x) = V^*(1\otimes x)V$ for $x\in L^\infty(\check{\G})$.
Then the pair $(L^\infty(\check{\G}), \check{\Delta})$ defines a locally compact quantum group $\check{\G}$. We have $L^\infty(\hat{\G})= L^\infty(\check{\G})'$. If $G$ is a locally compact group, we have
$L^\infty(\hat{G}) = \mathscr{L}(G)$ and $L^\infty(\check{G}) = \mathscr{R}(G)$, the left resp. right group von Neumann algebra associated with $G$. Therefore, $\hat{\G}$ and $\check{\G}$ should be thought of as a left and a right version of the dual locally compact quantum group.

The left multiplicative unitary of $\check{\G}$ will be denoted by $\check{W}$ and the right multiplicative unitary of $\check{\G}$ will be denoted by $\check{V}$. We have $\check{W}= V$ and 
$$V \in L^\infty(\check{\G})\ovot L^\infty(\G), \quad W \in L^\infty(\G)\ovot L^\infty(\hat{\G}), \quad \check{V}\in L^\infty(\G)'\ovot L^\infty(\check{\G}).$$
In fact, we have e.g.\ $V\in M(C_0(\check{\G})\otimes C_0(\G))$ and similarly for $W$ and  $\check{V}$. For future use, we will also need the following notations:
\begin{align*}
    &\Delta_l: B(L^2(\G)) \to  L^\infty(\G)\ovot B(L^2(\G)): x \mapsto W^*(1\otimes x)W,\\
    &\Delta_r: B(L^2(\G))\to B(L^2(\G))\ovot L^\infty(\G): x \mapsto V(x\otimes 1)V^*,\\
    &\check{\Delta}_r: B(L^2(\G))\to B(L^2(\G)) \ovot L^\infty(\check{\G}): x \mapsto \check{V}(x\otimes 1)\check{V}^*.
\end{align*}
A locally compact quantum group $\G$ is called \emph{compact} if $C_0(\G)$ is unital, and we then write $C_0(\G)= C(\G)$. In that case the left Haar weight $\Phi$ is a normal state and we say that $\G$ is of \emph{Kac type} if the Haar state $\Phi: L^\infty(\G)\to \mathbb{C}$ is tracial. A locally compact quantum group $\G$ is called \emph{discrete} if $\check{\G}$ is compact and we then write $L^\infty(\G)= \ell^\infty(\G)$ and $C_0(\G)= c_0(\G)$. If $\G$ is a discrete quantum group, there is a unique normal state $\epsilon \in \ell^1(\G):= L^1(\G)$, called \emph{counit}, such that 
$(\epsilon \otimes \id)\Delta = \id = (\id \otimes \epsilon)\Delta$. We call a discrete quantum group $\G$ \emph{unimodular} if $\check{\G}$ is of Kac type. A locally compact quantum group $\G$ is called \emph{co-amenable} \cite{BT03} if there exists a state $\epsilon \in C_0(\G)^*$ such that $(\id \otimes \epsilon) \Delta = \id_{C_0(\G)} = (\epsilon\otimes \id)\Delta$, where we note that this makes sense since we have a canonical extension $\id \otimes \epsilon: M(C_0(\G)\otimes C_0(\G)) \to M(C_0(\G))$. Thus, it is clear that discrete quantum groups are always co-amenable. On the other hand, a locally compact quantum group $\G$ is called \emph{amenable} if there exists a state $m: L^\infty(\G)\to \mathbb{C}$ such that 
$$(m\otimes \id)\Delta(x)= m(x)1, \quad x \in L^\infty(\G).$$
If $\G$ is co-amenable, then $\check{\G}$ is amenable but the converse is still open for general locally compact quantum groups. We say that a locally compact quantum group $\G$ is \emph{inner amenable} \cite{Cr19} if there exists a state $n: L^\infty(\check{\G})\to \mathbb{C}$ such that 
$$(n\otimes \id)(\Delta_r(x)) = n(x)1, \quad x \in L^\infty(\check{\G}).$$ If $G$ is a locally compact group, this is equivalent with the existence of a conjugation-invariant mean on $L^\infty(G)$ \cite[Proposition 3.2]{CZ17}. 

We will need the following elementary lemma several times:

\begin{Lem}\label{well-defined} Let $\G$ be a locally compact quantum group and $\mathcal{H}$ a Hilbert space.
    If $U\in B(\mathcal{H})\ovot L^\infty(\G)$ is a unitary satisfying $(\id \otimes \Delta)(U)= U_{12}U_{13}$, then
    $$U^*(1\otimes L^\infty(\check{\G}))U\subseteq B(\mathcal{H})\ovot L^\infty(\check{\G}).$$
\end{Lem}
\begin{proof}
    By the assumption, we see that $V_{23}U_{12}V_{23}^* = U_{12}U_{13}$ or equivalently
    $$U_{12}^*V_{23} U_{12}= U_{13}V_{23}\in B(\mathcal{H})\ovot L^\infty(\check{\G})\ovot B(L^2(\G)).$$
    Therefore, if $\omega \in B(L^2(\G))_*$, we find
    \begin{align*}
        U^*(1\otimes (\id \otimes \omega)(V))U = (\id \otimes \id \otimes \omega)(U_{12}^* V_{23} U_{12}) = (\id \otimes \id \otimes \omega)(U_{13}V_{23})\in B(\mathcal{H})\ovot L^\infty(\check{\G}).
    \end{align*}
    Recalling that $L^\infty(\check{\G})=[(\id \otimes \omega)(V): \omega \in B(L^2(\G))_*]^{\sigma\textrm{-weak}}$, the lemma follows.
\end{proof}

\subsection{Actions of locally compact quantum groups and crossed products.}

A (right) \emph{$\G$-dynamical von Neumann algebra} is a pair $(M, \alpha)$ such that $M$ is a von Neumann algebra and $\alpha: M \to M\ovot L^\infty(\G)$ is an injective, unital, normal $*$-homomorphism satisfying the coaction property $(\alpha\otimes \id)\circ \alpha = (\id \otimes \Delta)\circ \alpha$. We sometimes denote this with $\alpha: M\curvearrowleft\G$. We write
$M^{\alpha} = \{x\in M: \alpha(x)= x \otimes 1\}$
for the von Neumann subalgebra of fixed points of $(M, \alpha)$ and the trivial $\G$-action on a von Neumann algebra $M$ will always be denoted by $\tau$, i.e. $\tau(x)= x \otimes 1$ for $x\in M$. Given a $\G$-dynamical von Neumann algebra $(M, \alpha)$,  we define the \emph{crossed product von Neumann algebra}
$M\rtimes_\alpha \G= [\alpha(M)(1\otimes L^\infty(\check{\G}))]''.$
We have the alternative description
$$M\rtimes_\alpha \G = \{z \in M \ovot B(L^2(\G)): (\alpha\otimes \id)(z) =(\id \otimes \Delta_l)(z)\}.$$
It is easy to verify that we have maps
$$\id \otimes \Delta_r: M \rtimes_\alpha \G \to (M\rtimes_\alpha \G)\ovot L^\infty(\G), \quad \id \otimes \check{\Delta}_r: M\rtimes_\alpha \G \to (M\rtimes_\alpha \G)\ovot L^\infty(\check{\G})$$
so that $M\rtimes_\alpha \G$ becomes a $\G$-resp.~ $\check{\G}$-dynamical von Neumann algebra. We have that $(M\rtimes_\alpha \G)^{\id \otimes \check{\Delta}_r}= \alpha(M)$. Given any von Neumann algebra $M$, we have $M\rtimes_\tau \G = M\ovot L^\infty(\check{\G})$. Given two $\G$-dynamical von Neumann algebras $(M, \alpha)$ and $(N, \beta)$, a completely bounded map $\phi: M\to N$ is called $\G$-equivariant if $(\phi\otimes \id)\circ \alpha= \beta \circ \phi.$

There is a unique $*$-isomorphism $\kappa: L^\infty(\G)\rtimes_\Delta \G \to B(L^2(\G))$ such that $\kappa(\Delta(x)) = x$ and $\kappa(1\otimes y)= y$ for all $x\in L^\infty(\G)$ and all $y\in L^\infty(\check{\G})$. In fact, $\kappa^{-1}= \Delta_l: B(L^2(\G))\to L^\infty(\G)\rtimes_\Delta \G.$ It defines a $\G$-equivariant resp.~ $\check{\G}$-equivariant isomorphism 
        $$(L^\infty(\G)\rtimes_\Delta \G, \id \otimes \Delta_r)\cong (B(L^2(\G)), \Delta_r), \qquad  (L^\infty(\G)\rtimes_\Delta \G, \id \otimes \check{\Delta}_r)\cong (B(L^2(\G)), \check{\Delta}_r).$$

If $(M, \alpha)$ is a $\G$-dynamical von Neumann algebra where $M$ is standardly represented on a Hilbert space $\mathcal{H}$, there exists a canonical unitary $U_\alpha \in B(\mathcal{H})\ovot L^\infty(\G)$ such that $\alpha(m)= U_\alpha(m\otimes 1)U_\alpha^*$ for all $m\in M$ \cite{Va01}. We call $U_\alpha$ the \emph{unitary implementation} of $\alpha$, and it satisfies $(\id \otimes \Delta)(U_\alpha)= U_{\alpha,12}U_{\alpha,13}.$

Following the terminology introduced in \cite{DD24}, we say that $(M, \alpha)$ is \emph{$\G$-amenable} if there exists a $\G$-equivariant unital completely positive conditional expectation $(M \ovot L^\infty(\G), \id \otimes \Delta)\to (\alpha(M), \id \otimes \Delta)$ and we call $(M, \alpha)$ \emph{inner $\G$-amenable} if there exists a $\G$-equivariant unital completely positive conditional expectation $(M\rtimes_\alpha \G, \id \otimes \Delta_r)\to (\alpha(M), \id \otimes \Delta)$. It is not so hard to see that the notion of an amenable action of a locally compact group on a von Neumann algebra, as introduced in \cite{AD79}, is equivalent to the quantum definition (in \cite{AD79}, the author uses left shifts, whereas we must use right shifts in our conventions). We also mention the important fact that $(M, \alpha)$ is $\G$-amenable if and only if $(M\rtimes_\alpha \G, \id \otimes \check{\Delta}_r)$ is inner $\check{\G}$-amenable and that $(M, \alpha)$ is inner $\G$-amenable if and only if $(M\rtimes_\alpha \G, \id \otimes \check{\Delta}_r)$ is $\check{\G}$-amenable \cite[Theorem 6.12]{DD24}. 

Let us also mention the following well-known principle:

\begin{Prop}\label{general result}
    Let $(M, \alpha)$ be a $\G$-dynamical von Neumann algebra and $N$ a von Neumann subalgebra of $M$ with $\alpha(N)\subseteq N \ovot L^\infty(\G)$, so that also $(N, \alpha)$ is a $\G$-dynamical von Neumann algebra. The following statements are equivalent:
    \begin{enumerate}
        \item There exists a $\G$-equivariant unital completely positive conditional expectation $\phi: M \to N$.
        \item There exists a $\check{\G}$-equivariant unital completely positive conditional expectation $\psi: M\rtimes_\alpha \G \to N \rtimes_\alpha \G$.
    \end{enumerate}
\end{Prop}
\begin{proof}
    $(1)\implies (2)$ We can take $\psi= \phi\otimes \id$.

    $(2)\implies (1)$ Since $\psi$ is $\check{\G}$-equivariant, it restricts to a map $\psi: \alpha(M)\to \alpha(N)$ between the fixed points. Then $\phi:= \alpha^{-1}\circ \psi\circ \alpha: M \to N$ is a unital completely positive conditional expectation. We have $\psi(1\otimes y)= 1 \otimes y$ for all $y\in L^\infty(\check{\G})$. From this, it follows that $(\psi\otimes \id)(z) = z$ for all $z\in M(1\otimes C_0(\check{\G})\otimes C_0(\G))$. In particular, $V_{23}$ belongs the multiplicative domain of $\psi\otimes \id$. Therefore, if $z\in M\rtimes_\alpha \G$, we have
    \begin{align*}
        (\psi\otimes \id)(\id \otimes \Delta_r)(z) = (\psi\otimes \id)(V_{23}(z\otimes 1)V_{23}^*)= V_{23}(\psi(z)\otimes 1)V_{23}= (\id \otimes \Delta_r)\psi(z).
    \end{align*}
From this, it follows that the map $\phi$ is $\G$-equivariant.
\end{proof}

\section{Equivariant operator spaces, injectivity and amenability}

We fix a locally compact quantum group $\G$.

\begin{Def}\label{Equivariant operator space} A (right) $\G$-operator space is a pair $(X, \alpha)$ where $X$ is an operator space and $\alpha: X \to X \ovot L^\infty(\G)$ is a complete isometry such that the diagram
$$
\begin{tikzcd}
X \arrow[d, "\alpha"'] \arrow[rr, "\alpha"]            &  & X \ovot L^\infty(\G) \arrow[d, "\id \otimes \Delta"] \\
X \ovot L^\infty(\G) \arrow[rr, "\alpha \otimes \id"'] &  & X \ovot L^\infty(\G)\ovot L^\infty(\G)              
\end{tikzcd}$$
commutes. Whenever convenient, we will write $X\stackrel{\alpha}\curvearrowleft\G$. We also write $X^\alpha = \{x\in X: \alpha(x)= x\otimes 1\}$ for the space of $\G$-fixed points of $(X, \alpha)$. A (right) $\G$-operator system is a $\G$-operator space $(X, \alpha)$ where $X$ is an operator system and $\alpha$ is unital.
\end{Def}

We give some remarks:
\begin{itemize}
\item In the terminology of \cite[Definition 2.1]{Ham91}, a $\G$-operator space is simply a comodule over the Hopf-von Neumann algebra $(L^\infty(\G), \Delta).$
\item We emphasize that $X$ in the above definition need not be a dual operator space, and even if it were, the action $\alpha$ need not be $\sigma$-weakly continuous.
\item A $\G$-dynamical von Neumann algebra is a $\G$-operator system $(M, \alpha)$ such that $M$ is a von Neumann algebra and $\alpha$ is a normal $*$-homomorphism. 
    \item In \cite[Definition 3.8, Definition 3.22]{DH24}, the notion of a $\G$-operator system was already introduced where $\G$ is a compact quantum group or a discrete quantum group (in fact, in \cite{DH24}, we used the terminology $\G$-$W^*$-operator system to make the distinction clear between the notion of $\G$-$C^*$-operator system and $\G$-$W^*$-operator system; in this paper, no confusion can arise so we omit the prefix $W^*$ from the terminology). Definition \ref{Equivariant operator space} is the obvious generalisation to the more general context of operator spaces and locally compact quantum groups. 
    \item  Let $(X, \alpha)$ be a right $\G$-operator space. Under the canonical completely isometric identification
$$CB(X, X \ovot L^\infty(\G))\cong CB(X, CB(L^1(\G),X))\cong CB(L^1(\G) \hat{\otimes} X, X)$$
the coaction $\alpha: X \to X \ovot L^\infty(\G)$ induces a left operator $L^1(\G)$-module structure $\rhd$ on $X$ given by
$$\omega\rhd x = (\id \otimes \omega)\alpha(x), \quad x \in X, \omega \in L^1(\G).$$Since $\alpha$ is a complete isometry, we see that the canonical map
$$j_X: X \to CB(L^1(\G), X), \quad j_X(x)(\omega) = \omega\rhd x, \quad x \in X, \omega \in L^1(\G)$$
is completely isometric. It follows that a right $\G$-operator space is the same thing as a left operator $L^1(\G)$-module such that the canonical map $j_X: X \to CB(L^1(\G),X)$ is completely isometric. 
    \item We can similarly define the notion of left $\G$-operator space/system. However, we will only deal with right $\G$-operator spaces in this paper, so from now on all $\G$-operator spaces are assumed to be right $\G$-operator spaces. 
\end{itemize}

We now show that $\G$-dynamical von Neumann algebras lead to important examples of $\G$-operator systems.
\begin{Exa} Let $M\stackrel{\alpha}\curvearrowleft\G$ be a $\G$-dynamical von Neumann algebra. Define the operator system
$$X:= [(\omega \otimes \id)\alpha(m): m \in M, \omega \in M_*]\subseteq L^\infty(\G).$$
Then clearly for $m\in M,  \omega \in M_*$ and $\eta \in B(L^2(\G))_*$,
\begin{align*}
    (\id \otimes \eta)\Delta((\omega \otimes \id)\alpha(m))&= (\id \otimes \eta)(\omega \otimes \id \otimes \id)(\id \otimes \Delta)\alpha(m)\\
    &= (\omega \otimes \id \otimes \eta)(\alpha \otimes \id)\alpha(m)\\
    &= (\omega \otimes \id)\alpha((\id \otimes \eta)\alpha(m)) \in X
\end{align*}
so that $\Delta((\omega \otimes \id)\alpha(m))\in X \bar{\otimes}L^\infty(\G)$. Thus, the map $\Delta$ restricts to a completely isometric, unital coaction
$$\Delta: X \to X \bar{\otimes} L^\infty(\G).$$
It is not clear if $X$ is a $C^*$-subalgebra of $L^\infty(\G)$ if $\G$ fails to be semiregular. To illustrate this, note that if $L^\infty(\G)\stackrel{\Delta}\curvearrowleft\G$, we find back the space of left uniformly continuous elements $X=LUC(\G)$ \cite{Ru09}. In \cite{HNR11}, it was shown that if $\G$ is semi-regular, then $LUC(\G)$ is a $C^*$-subalgebra of $M(C_0(\G))\subseteq L^\infty(\G)$. But in the general case, it is not clear that $X$ is a $C^*$-subalgebra. Note that the dynamics of $LUC(\G)$ was used in \cite[Proposition 5.8]{Cr21} to characterise the co-amenability of $\G$.

On the other hand, we can also consider
$$Y := [(\id \otimes \omega)\alpha(m): m \in M, \omega \in L^1(\G)]\subseteq M.$$

If $m\in M$ and $\omega, \eta \in L^1(\G)$, we have
\begin{align*}
    (\id \otimes \eta)\alpha((\id \otimes \omega)\alpha(m))&= (\id \otimes \eta)(\id \otimes\id \otimes \omega)(\alpha \otimes \id)\alpha(m)\\
    &= (\id \otimes \eta \otimes \omega)(\id \otimes \Delta)\alpha(m)\\
    &= (\id \otimes (\eta\star \omega))\alpha(m)\in Y
\end{align*}
so we conclude that the map $\alpha$ restricts to a completely isometric, unital coaction
$$\alpha: Y \to Y \bar{\otimes} L^\infty(\G).$$ Again, it is not clear that $Y$ carries a $C^*$-algebra structure outside the semiregular case. Note that the operator system $Y$ plays an important role in \cite{DD24}, although its dynamics was not explicitly considered.
\end{Exa}

\begin{Def}
    Given two $\G$-operator spaces $(X, \alpha)$ and $(Y, \beta)$, a completely bounded map $\phi: X \to Y$ is said to be $\G$-equivariant if the following diagram commutes:
    $$
\begin{tikzcd}
X \arrow[d, "\alpha"'] \arrow[rr, "\phi"]          &  & Y \arrow[d, "\beta"] \\
X\ovot L^\infty(\G) \arrow[rr, "\phi\otimes \id"'] &  & Y \ovot L^\infty(\G)
\end{tikzcd}$$
\end{Def}
\begin{Rem}
    In terms of the induced $L^1(\G)$-module structures, it is easy to see that $\phi: X \to Y$ is $\G$-equivariant if and only if $\phi(\omega \rhd x) = \omega \rhd \phi(x)$ for all $\omega \in L^1(\G)$ and all $x \in X$, i.e. $\phi$ is an $L^1(\G)$-module morphism.
\end{Rem}

\begin{Def}
    Let $(X, \alpha)$ be a $\G$-operator space (resp.~ system). \begin{itemize}
        \item $(X, \alpha)$ is said to be $\G$-injective  as an operator space (resp.~ system) if for all $\G$-operator spaces (resp.~ systems) $Y$ and $Z$ together with a $\G$-equivariant (resp.~ unital) complete contraction  $\phi: Y\to X$ and a $\G$-equivariant (resp.~ unital) complete isometry $\iota: Y\to Z$, there exists a $\G$-equivariant (resp.~ unital) complete contraction $\Phi: Z\to X$ such that $\Phi\circ \iota = \phi$. 
        \item $(X, \alpha)$ is said to be $\G$-amenable if there exists a $\G$-equivariant completely contractive conditional expectation $E: (X \ovot L^\infty(\G), \id \otimes \Delta)\to (\alpha(X), \id \otimes \Delta)$. 
    \end{itemize}
\end{Def}

We now show the existence of $\G$-injective envelopes. First, we give the following standard definitions:

\begin{Def}
    Let $X$ be a $\G$-operator space (resp.~system).
    \begin{itemize}
        \item A $\G$-extension of $X$ is a pair $(Y, \iota)$ where $Y= (Y, \beta)$ is a $\G$-operator space (resp.~system) and $\iota: X \to Y$ a $\G$-equivariant (resp.~~unital) complete isometry. We write $(X, \alpha)\subseteq_\G (Y, \beta)$ if $X\subseteq Y$ and $\beta\vert_X = \alpha$.
        \item A $\G$-extension $(Y, \iota)$ of $X$ is called $\G$-rigid if for every $\G$-equivariant (resp.~ unital) completely contractive map $\phi: Y \to Y$ with $\phi\iota = \iota$, we have $\phi = \id_Y$.
        \item A $\G$-extension $(Y, \iota)$ of $X$ is called $\G$-essential if every $\G$-equivariant (resp.~ unital) completely contractive map $\phi: Y \to Z$ is a complete isometry whenever $\phi\circ \iota: X \to Z$ is a complete isometry.
        \item A $\G$-injective extension $(Y, \iota)$ of $X$ is called $\G$-injective envelope if the situation $\iota(X)\subseteq_\G \widetilde{X}\subseteq_\G Y$ with $\widetilde{X}$ $\G$-injective implies that $Y = \widetilde{X}$.
    \end{itemize}
\end{Def}

Note that if $(X, \alpha)$ is a $\G$-operator space and $X\subseteq B(\mathcal{H})$, then the map
$\alpha: (X, \alpha) \to  (B(\mathcal{H})\ovot L^\infty(\G), \id \otimes \Delta)$ is a $\G$-equivariant complete isometry. In particular, there exists a $\G$-dynamical von Neumann algebra $(Y, \beta)$ such that $X\subseteq_\G Y$. We will frequently exploit this fact to reduce problems to the situation of $\G$-dynamical von Neumann algebras, where things are usually much easier to handle.

 \begin{Prop}\label{injective envelopes}
    If $(X, \alpha)$ is a $\G$-operator space (resp.~ system), there exists a $\G$-extension $(Y, \kappa: X \to Y)$ that is a $\G$-injective envelope of $X$. It is unique in the sense that if $(\widetilde{Y}, \widetilde{\kappa}: X \to \widetilde{Y})$ is another $\G$-injective envelope for $X$, then there is a unique $\G$-equivariant (resp.~ unital) completely isometric isomorphism $\Phi: Y \to \widetilde{Y}$ such that the diagram
    $$
\begin{tikzcd}
                              & X \arrow[ld, "\kappa"', hook] \arrow[rd, "\widetilde{\kappa}", hook'] &               \\
Y \arrow[rr, "\Phi"', dashed] &                                                                       & \widetilde{Y}
\end{tikzcd}$$
commutes. Moreover, a $\G$-extension $(Y, \kappa: X \to Y)$ is a $\G$-injective envelope if and only if the extension is $\G$-injective and $\G$-rigid, and in this case the extension is also $\G$-essential.
\end{Prop}
\begin{proof} First, we treat the case where $X$ is a $\G$-operator space. Under the completely isometric isomorphism $$CB(X, X \ovot L^\infty(\G))\cong CB(X, CB(L^1(\G),X))\cong CB(L^1(\G)\hat{\otimes} X, X)$$
the coaction $\alpha: X \to X \ovot L^\infty(\G)$ induces a complete contraction $\rhd: L^1(\G)\hat{\otimes}X \to X$ that turns $X$ into a (left) operator module over the completely contractive Banach algebra $L^1(\G)$. Let now $(Y, \iota: X \to Y)$ be the injective envelope in the category of operator $L^1(\G)$-modules (with completely contractive $L^1(\G)$-linear maps as morphisms) \cite[Theorem 2.1]{BC21} and consider its associated module map $\RHD: L^1(\G)\hat{\otimes} Y \to Y$ and the canonical map
$$j_Y: Y \to CB(L^1(\G), Y), \quad j_Y(y)(\omega)= \omega \RHD y.$$
Then clearly the diagram 
$$
\begin{tikzcd}
X \arrow[rr, "\iota"] \arrow[d, "j_X"']                             &  & Y \arrow[d, "j_Y"] \\
{CB(L^1(\G),X)} \arrow[rr, "\hat{\iota}: f \mapsto \iota \circ f"'] &  & {CB(L^1(\G),Y)}   
\end{tikzcd}$$
commutes. Since $j_X: X \to CB(L^1(\G),X)$ and $\hat{\iota}: CB(L^1(\G),X)\to CB(L^1(\G),Y)$ are complete isometries, the $L^1(\G)$-essentiality of $Y$ \cite[Theorem 2.1]{BC21} implies that also $j_Y$ is a complete isometry. But under the completely isometric isomorphism
$$CB(Y, Y\ovot L^\infty(\G))\cong CB(Y, CB(L^1(\G), Y))\cong CB(L^1(\G)\hat{\otimes}Y, Y)$$
the module map $\RHD$ therefore induces a completely isometric coaction $\beta: Y \to Y \ovot L^\infty(\G)$. Thus, $Y$ is a $\G$-operator space for the action $\beta: Y \curvearrowleft \G$. It is then easy to see $(Y, \iota)$ is a $\G$-injective envelope for $X$.

Next, we treat the case where $X$ is a $\G$-operator system. The argument is similar, but we can no longer use the injective envelope in the category of operator $L^1(\G)$-modules. Therefore, we will construct the injective envelope in an appropriate subcategory of operator $L^1(\G)$-modules. We say that an operator $L^1(\G)$-module $(Z, \rhd)$ is an $L^1(\G)$-module operator system if $Z$ is an operator system and $\omega \rhd 1 = \omega(1)1$ for all $\omega \in L^1(\G)$ (on the level of coactions, this module condition means unitality of the coaction). We embed $X\subseteq B(\mathcal{H})$ and we note that there is a canonical completely isometric identification
$$N:=(B(\mathcal{H})\ovot L^\infty(\G))\oplus_\infty B(\mathcal{H}) \cong CB(L^1(\G)_1, B(\mathcal{H})):(z,x)\mapsto [\omega + \lambda e \mapsto (\id \otimes \omega)(z)+ \lambda x].$$
Thus, it is clear that $CB(L^1(\G)_1, B(\mathcal{H}))$ admits the structure of an operator system. The unit $\mathbb{I}\in CB(L^1(\G)_1, B(\mathcal{H}))$ is given by
$$\mathbb{I}(\omega + \lambda e)= (\omega(1)+ \lambda)1, \quad \omega + \lambda e \in L^1(\G)_1$$
and corresponds to the unit of the von Neumann algebra $N$. For the canonical $L^1(\G)$-module action on $CB(L^1(\G)_1, B(\mathcal{H}))$, the space $CB(L^1(\G)_1, B(\mathcal{H}))$ then becomes an $L^1(\G)$-module operator system and an obvious modification of the argument in the proof of \cite[Proposition 2.3]{Cr17a} shows that $CB(L^1(\G)_1, B(\mathcal{H}))$ is injective in the category of $L^1(\G)$-module operator systems and $L^1(\G)$-linear unital complete contractions (= unital completely positive maps). Therefore, if we view $N$ with the $L^1(\G)$-module structure
$$\omega \rhd (z,x) = ((\id \otimes \id \otimes \omega)(\id \otimes \Delta)(z), (\id \otimes \omega)(z)), \quad \omega \in L^1(\G), (z,x)\in N$$
we see that $N\cong CB(L^1(\G)_1, B(\mathcal{H}))$ as $L^1(\G)$-module operator systems. Thus, $N$ is injective in the category of $L^1(\G)$-module operator systems.

Consider the $L^1(\G)$-linear unital complete isometry
$$\kappa: X \to N: x \mapsto (\alpha(x), x)$$
and the set $\mathcal{G}$ of all unital completely positive $L^1(\G)$-linear maps $\phi: N \to N$ that satisfy $\phi\circ \kappa = \kappa$. Then $\mathcal{G}$ is closed in the topology of pointwise $\sigma$-weak convergence: if $\{\phi_i\}_{i\in I}\subseteq \mathcal{G}$ and if $\phi: N \to N$ is a unital completely positive map such that $\phi_i\to \phi$ in the point $\sigma$-weak topology, then we have for $\omega \in L^1(\G)$ and $n\in N$, and with $\pi_1: N \to B(\mathcal{H}) \ovot L^\infty(\G)$ the projection map, that 
\begin{align*}
    \phi(\omega \rhd n) &= \sigma\text{-}\lim_{i\in I} \phi_i(\omega \rhd n)\\
    &= \sigma\text{-}\lim_{i\in I} \omega \rhd \phi_i(n)\\
    &= \sigma\text{-}\lim_{i\in I} ((\id \otimes \id \otimes \omega)(\id \otimes \Delta)\pi_1(\phi_i(n)), (\id \otimes \omega)\pi_1(\phi_i(n)))\\
    &= ((\id \otimes \id \otimes \omega)(\id \otimes \Delta)\pi_1(\phi(n)), (\id \otimes \omega)\pi_1(\phi(n)))= \omega \rhd \phi(n)
\end{align*}
and thus $\phi$ is $L^1(\G)$-linear. Clearly also $\phi\circ \kappa = \kappa$, so $\phi\in \mathcal{G}$. 

By \cite[Proposition 2.1]{HHN22}, we find an idempotent $\varphi\in \mathcal{G}$ such that $\varphi = \varphi \phi  \varphi$ for all $\phi\in \mathcal{G}$. Since $\varphi \kappa = \kappa$, we see that $\kappa(X)\subseteq \varphi(N)$. By a standard argument (see e.g.\ the proof of \cite[Proposition 5.3]{DH24}), the pair
$(\varphi(N), \kappa: X \to \varphi(N))$
is an injective envelope of $X$ in the category of $L^1(\G)$-module operator systems. By the same argument as in the first part of the proof (using $L^1(\G)$-essentiality), the pair $(\varphi(N), \kappa)$ is also an injective envelope of $X$ in the category of $\G$-operator systems.
\end{proof}

If $X$ is a $\G$-operator space, its $\G$-injective envelope (as a $\G$-operator space) is denoted by $I_\G(X)$. The above proof shows that $I_\G(X)= I_{L^1(\G)}(X)$ (with the notation on the right denoting the injective envelope in the category of 
operator $L^1(\G)$-modules and completely contractive $L^1(\G)$-linear maps). If $X$ is a $\G$-operator system, its injective envelope (as a $\G$-operator system) is denoted by $I_\G^1(X).$  

The attentive reader may have noticed that the proof of Proposition \ref{injective envelopes} did not make use of the invariant weights on the Hopf-von Neumann algebra $(L^\infty(\G), \Delta)$. In other words, if $(M, \Delta)$ is a Hopf-von Neumann algebra and the notions of $M$-operator space/system, $M$-injectivity, etc. are defined in the obvious way, we find the following general result:

\begin{Prop}\label{injective envelopes3} Let $(M, \Delta)$ be a Hopf-von Neumann algebra. If $(X, \alpha)$ is an $M$-operator space (resp.~~system), there exists an $M$-extension $(Y, \kappa: X \to Y)$ that is an $M$-injective envelope of $X$. It is unique in the sense that if $(\widetilde{Y}, \widetilde{\kappa}: X \to \widetilde{Y})$ is another $M$-injective envelope for $X$, then there is a unique $M$-equivariant (resp.~ unital) completely isometric isomorphism $\Phi: Y \to \widetilde{Y}$ such that the diagram
    $$
\begin{tikzcd}
                              & X \arrow[ld, "\kappa"', hook] \arrow[rd, "\widetilde{\kappa}", hook'] &               \\
Y \arrow[rr, "\Phi"', dashed] &                                                                       & \widetilde{Y}
\end{tikzcd}$$
commutes. Moreover, an $M$-extension $(Y, \kappa: X \to Y)$ is an $M$-injective envelope if and only if the extension is $M$-injective and $M$-rigid, and in this case the extension is also $M$-essential.
\end{Prop}

\begin{Rem}
   The $M$-operator space version of Proposition \ref{injective envelopes} is also proven in \cite[Theorem 2.7]{Ham91}, under the additional assumption that $M_*$ admits an approximate unit $\{u_i\}_{i\in I}$ such that $\lim_{i\in I}\|u_i\| = 1$. In other words, our proof shows that \cite[Theorem 2.7]{Ham91} is still true when we omit the `co-amenability' condition on $(M, \Delta)$.
\end{Rem}

In the following result, we clarify the relation between different notions of $\G$-equivariant injectivity. 

\begin{Prop}\label{injective}
    Let $(X, \alpha)$ be a $\G$-operator space. The following statements are equivalent:
    \begin{enumerate}
    \item $X$ is injective as an $L^1(\G)$-module.
        \item $(X, \alpha)$ is $\G$-injective as a $\G$-operator space.
        \item $X$ is injective (as an operator space) and $(X, \alpha)$ is $\G$-amenable.
    \end{enumerate}
    If moreover $(X, \alpha)$ is a $\G$-operator system, then these conditions are also equivalent with:
    \begin{enumerate}
        \item[(4)] $(X, \alpha)$ is $\G$-injective as a $\G$-operator system.
    \end{enumerate}
\end{Prop}

\begin{proof} $(1)\implies (2)$ Trivial. 

$(2)\implies (3)$ We may assume that $X\subseteq B(\mathcal{H})$. Then the map $\alpha: (X, \alpha)\to (B(\mathcal{H})\bar{\otimes} B(L^2(\G)), \id \otimes \Delta_r)$ is $\G$-equivariant. By $\G$-injectivity of $X$, there is a $\G$-equivariant completely contractive map $E: B(\mathcal{H})\bar{\otimes} B(L^2(\G))\to X$ such that the diagram
$$
\begin{tikzcd}
X \arrow[rr, "\alpha"] \arrow[d, "\id_X"] &  & X\bar{\otimes}L^\infty(\G) \arrow[rr, "\subseteq"] &  & B(\mathcal{H})\bar{\otimes} B(L^2(\G)) \arrow[lllld, "E", dashed] \\
X                                         &  &                                                    &  &                                                                  
\end{tikzcd}$$
commutes, i.e. $E(\alpha(x)) = x$ for all $x\in X$. Since $\alpha$ is completely isometric and $B(\mathcal{H})\ovot B(L^2(\G)) = B(\mathcal{H}\otimes L^2(\G))$ is injective as an operator space, we conclude that $\alpha(X)\cong X$ is injective as an operator space as well. The fact that $(X, \alpha)$ is $\G$-amenable is immediate from the $\G$-injectivity of $(X, \alpha)$, since the map $\alpha: (X, \alpha)\to (X\ovot L^\infty(\G), \id \otimes \Delta)$ is $\G$-equivariant.

$(3)\implies (1)$ Since $(X, \alpha)$ is $\G$-amenable, there exists a $\G$-equivariant complete contraction $$E: (X \ovot L^\infty(\G), \id \otimes \Delta)\to (X, \alpha)$$ such that $E\circ \alpha = \id_X$. Assume that $Y,Z$ are operator $L^1(\G)$-modules, $\iota: Y \to Z$ an $L^1(\G)$-linear complete isometry and $\phi: Y \to X$ an $L^1(\G)$-linear complete contraction. Since $X$ is injective as an operator space, there exists a completely contractive map $\theta: Z \to X$ such that $\theta \circ \iota = \phi$. We then use amenability of the action $\alpha$ to make $\theta$ $\G$-equivariant. Concretely, let $\gamma: Z \to Z \ovot L^\infty(\G)$ be the image of the module map on $Z$ under the completely isometric isomorphism
$$CB(L^1(\G)\hat{\otimes}Z, Z)\cong CB(Z, Z \ovot L^\infty(\G))$$
(it is possible that $\gamma$ is not a complete isometry, but this is not important for the argument) and define the complete contraction
$$\Phi: = E\circ(\theta \otimes \id) \circ \gamma: Z\to X.$$
Then it is easily verified that $\Phi$ is an $L^1(\G)$-linear complete contraction such that $\Phi\circ \iota = \phi$.

Next, let $(X, \alpha)$ be a $\G$-operator system. 

$(4)\implies (3)$ The same argument as in $(2)\implies (3)$ works.

$(3)\implies (4)$ The same argument as in $(3)\implies (1)$ works.
\end{proof}

\begin{Rem}
    It is not clear that if $(M, \Delta)$ is a general Hopf-von Neumann algebra and $(X, \alpha)$ an $M$-injective operator space, that $X$ is also injective as an operator space.
\end{Rem}
Since $\mathbb{C}$ is injective and $\G$ is amenable if and only if the trivial action $\tau:\mathbb{C}\curvearrowleft\G$ is amenable, we immediately find that:

\begin{Cor}\label{amenab}
$\G$ is amenable if and only if $(\mathbb{C}, \tau)$ is $\G$-injective.    
\end{Cor}

A $\G$-dynamical von Neumann algebra $(M, \alpha)$ is called $\G$-$W^*$-amenable \cite[Definition 6.6]{DD24} if there exists a $\G$-equivariant unital completely positive map $$E: (B(L^2(M))\ovot L^\infty(\G), \gamma)\to (M, \alpha)$$ such that $E(\pi_M(m) \otimes 1) = m$ for all $m\in M$, where $\pi_M: M \to B(L^2(M))$ is the standard representation of $M$ and where the coaction $\gamma: B(L^2(M))\ovot L^\infty(\G)\to B(L^2(M))\ovot L^\infty(\G) \ovot L^\infty(\G)$ is given by $$\gamma(z)= U_{\alpha, 13}(\id \otimes \Delta)(z)U_{\alpha,13}^*=U_{\alpha, 13}V_{23}z_{12}V_{23}^*U_{\alpha,13}^*.$$

\begin{Prop}\label{GWamenable} $(M, \alpha)$ is $\G$-$W^*$-amenable if and only if $(M, \alpha)$ is $\G$-injective.
\end{Prop}
\begin{proof} It is straightforward to see that the map $$(M, \alpha)\to (B(L^2(M))\ovot L^\infty(\G), \gamma): m\mapsto \pi_M(m)\otimes 1$$ is $\G$-equivariant. Therefore, it immediately follows that if $(M, \alpha)$ is $\G$-injective, then also $(M, \alpha)$ is $\G$-$W^*$-amenable.

Conversely, by \cite[Theorem 6.9]{DD24}, we know that $(M, \alpha)$ is $\G$-$W^*$-amenable if and only if $(M, \alpha)$ is $\G$-amenable and there exists a $\G$-equivariant unital completely positive conditional expectation $E: B(L^2(M))\to \pi_M(M)\cong M$. The existence of the latter conditional expectation entails that $M$ is injective. Thus, by Proposition \ref{injective}, we see that $(M, \alpha)$ is $\G$-injective.
\end{proof}

\section{Crossed products}

Once again, we fix a locally compact quantum group $\G$. We introduce crossed products and investigate basic properties.

\begin{Def}
    If $(X, \alpha)$ is a $\G$-operator space, we define the Fubini crossed product as the operator space
    $$X\rtimes_\alpha \G := \{z \in X \ovot B(L^2(\G)): (\id \otimes \Delta_l)(z) = (\alpha \otimes \id)(z)\}.$$
\end{Def}

Obviously, if $(X, \alpha)$ is a $\G$-operator system, then $X\rtimes_\alpha \G$ is also an operator system. If $(X, \alpha)$ is a $\G$-dynamical von Neumann algebra, then $X\rtimes_\alpha \G$ is the usual $W^*$-crossed product and the notation is unambiguous. As expected, the Fubini crossed product carries canonical $\G$ and $\check{\G}$-actions:

\begin{Prop}
    We have
    $$(\id \otimes \Delta_r)(X\rtimes_\alpha \G)\subseteq (X\rtimes_\alpha\G)\ovot L^\infty(\G), \quad (\id \otimes \check{\Delta}_r)(X\rtimes_\alpha\G)\subseteq (X\rtimes_\alpha \G)\ovot L^\infty(\check{\G}).$$
    Consequently, $(X\rtimes_\alpha \G, \id \otimes \Delta_r)$ is a $\G$-operator space and $(X\rtimes_\alpha \G, \id \otimes \check{\Delta}_r)$ is a $\check{\G}$-operator space.
\end{Prop}

\begin{proof} Similar as in \cite[Theorem 3.40, Proposition 3.43]{DH24}, making use of the commutation relations
$(\Delta_l\otimes \id)\circ\Delta_r=(\id \otimes \Delta_r)\circ \Delta_l$ and $(\Delta_l\otimes \id)\circ \check{\Delta}_r= (\id \otimes \check{\Delta}_r)\circ \Delta_l.$
\end{proof}

If $(X,\alpha)$ and $(Y, \beta)$ are $\G$-operator spaces, and if $\phi: (X,\alpha)\to (Y, \beta)$ is $\G$-equivariant, it is easily verified that 
$(\phi\otimes \id)(X\rtimes_\alpha\G)\subseteq Y\rtimes_\beta \G.$ We denote the restricted map by $\phi\rtimes \G$ and clearly
$$\phi\rtimes \G: (X\rtimes_\alpha \G, \id \otimes \check{\Delta}_r)\to (Y\rtimes_\beta \G, \id \otimes \check{\Delta}_r)$$
is $\check{\G}$-equivariant. Therefore, we may view the Fubini crossed product as a functor from the category of $\G$-operator spaces to the category of $\check{\G}$-operator spaces. 

So far, the Fubini crossed product behaved exactly like the von Neumann crossed product. However, there are some pitfalls that one has to be careful with. Namely:

\begin{itemize}
    \item It may happen that $\alpha(X)\subsetneq (X\rtimes_\alpha \G)^{\id \otimes \check{\Delta}_r}$.
    \item The Takesaki-Takai duality $(X\rtimes_\alpha \G)\rtimes_{\id \otimes \check{\Delta}_r} \check{\G} \cong X \ovot B(L^2(\G))$ may not hold.
\end{itemize}

The notion of $\G$-completeness of a $\G$-operator space, originally introduced in \cite[Definition 5.1]{Ham91} in the context of locally compact groups, lies at the heart of these issues. Hamana's definition readily generalises to the context of locally compact quantum groups:
\begin{Def}\label{com}
    A $\G$-operator space $X$ is said to be $\G$-complete if for every $\G$-extension $Y$ of $X$ and every $y\in Y$, the condition $L^1(\G) \rhd y \subseteq X$ implies that $y\in X$.
\end{Def}

Note that if $\G$ is a discrete quantum group, then every $\G$-operator space is $\G$-complete. Indeed, with $\epsilon\in \ell^1(\G)$ the counit, we have $\epsilon\rhd y = y$ for all $y\in Y$. Not every $\G$-operator space is $\G$-complete: Indeed, Hamana observes that if $G$ is a locally compact group and if $X$ is $G$-complete, then $X$ must be necessarily translation-invariant \cite[Proposition 5.6 (vi)]{Ham91} in the sense of \cite[Definition 3.2]{Ham91}. However, Hamana also provides a way to construct explicit examples of $G$-operator spaces that are not translation-invariant.

The following proposition contains all we need to know about $\G$-completeness for our purposes. In the proof, we will make use of the following simple fact: if $(X, \alpha)\subseteq_\G (Y, \beta)$ are $\G$-operator spaces, then
$$X\rtimes_\alpha \G = (Y\rtimes_\beta \G)\cap (X \ovot B(L^2(\G))).$$
\begin{Prop}\label{completeness} Let $(X, \alpha)$ be a $\G$-operator space.\begin{enumerate}
    \item If $(X, \alpha)$ is $\G$-injective, it is $\G$-complete.
    \item $(X, \alpha)$ is $\G$-complete if and only if $(X\rtimes_\alpha \G)^{\id \otimes \check{\Delta}_r}= \alpha(X)$.
    \item If $(X, \alpha)$ is $\G$-complete, the complete isometry
    $$\Phi: X \bar{\otimes} B(L^2(\G)) \to X\bar{\otimes} B(L^2(\G))\bar{\otimes} B(L^2(\G)): z \mapsto V_{23}^*(\alpha\otimes \id)(z)V_{23}.$$
    has image $\operatorname{\Phi}(X\ovot B(L^2(\G)))= (X\rtimes_\alpha \G)\rtimes_{\id \otimes \check{\Delta}_r} \check{\G}$. In other words, the Takesaki-Takai duality holds.
\end{enumerate}
\end{Prop}

\begin{proof} $(1)$ Assume that $X\subseteq_\G Y$. The $\G$-injectivity of $X$ yields a $\G$-equivariant complete contraction $\varphi: Y \to X$ such that $\varphi(x)=x$ for all $x\in X$. Assume that $y\in Y$ with $L^1(\G)\rhd y\subseteq X$. If $\omega \in L^1(\G)$, we have
    $$\omega \rhd \varphi(y) = \varphi(\omega \rhd y) = \omega \rhd y$$
    so that $y= \varphi(y)\in X$. 

    (2) Assume first that $(X, \alpha)$ is $\G$-complete. Choose a $\G$-dynamical von Neumann algebra $(Y, \beta)$ such that $(X, \alpha)\subseteq_\G (Y, \beta)$ and fix $z\in (X\rtimes_\alpha \G)^{\id \otimes \check{\Delta}_r}$. Clearly
     $$(X\rtimes_\alpha \G)^{\id \otimes \check{\Delta}_r}\subseteq (Y \rtimes_\beta \G)^{\id \otimes \check{\Delta}_r}= \beta(Y).$$ Thus, there is $y \in Y$ with $z=\beta(y)$ and we have $\omega \rhd y \in X$ for all $\omega\in L^1(\G)$. Therefore, the $\G$-completeness of $X$ ensures that $y\in X$. Thus, $z= \beta(y) = \alpha(y) \in \alpha(X)$.
     
Conversely, assume that $(X\rtimes_\alpha \G)^{\id \otimes \check{\Delta}_r}= \alpha(X)$. Consider any $\G$-extension $(X, \alpha)\subseteq_\G (Y, \beta)$. We observe that 
$$\alpha(X) = (X \ovot L^\infty(\G))\cap \beta(Y).$$
Indeed, if $z\in (X\ovot L^\infty(\G))\cap \beta(Y)$, then trivially $(\id \otimes \check{\Delta}_r)(z) = z\otimes 1$ and also
$$z \in (X \ovot B(L^2(\G)))\cap (Y\rtimes_\beta \G) = X\rtimes_\alpha \G.$$
Consequently, $z \in (X\rtimes_\alpha \G)^{\id \otimes \check{\Delta}_r}= \alpha(X)$.
If $y \in Y$ satisfies $L^1(\G)\rhd y \subseteq X$, then $\beta(y) \in X \ovot L^\infty(\G)$ and thus $\beta(y)\in \alpha(X)= \beta(X)$. By injectivity of $\beta$, we conclude that $y \in X$. Since $(Y, \beta)$ was arbitrary, we conclude that $(X, \alpha)$ is $\G$-complete.

     (3) The following proof is contained in \cite[Proposition 5.7]{Ham91}. We recall it for the convenience of the reader. Let $(Y, \beta)$ be any $\G$-dynamical von Neumann algebra such that $(X, \alpha)\subseteq_\G (Y, \beta)$. Then we know that the map
    $$\Phi: Y\bar{\otimes} B(L^2(\G)) \to (Y\rtimes_\beta\G)\rtimes_{\id \otimes \check{\Delta}_r} \check{\G}: z\mapsto V_{23}^*(\beta\otimes \id)(z)V_{23}$$
    is a $*$-isomorphism of von Neumann algebras (see e.g.\ \cite[Theorem 5.24]{DD24} or \cite[Theorem 2.6]{Va01}).
    Note first that
    \begin{align*}
        (X\rtimes_\alpha \G)\rtimes_{\id \otimes \check{\Delta}_r} \check{\G}&= [((X\bar{\otimes} B(L^2(\G))) \cap (Y\rtimes_\beta \G))\bar{\otimes} B(L^2(\G))]\cap [(Y\rtimes_\beta \G)\rtimes_{\id \otimes \check{\Delta}_r}\check{\G}]\\
        &= [X\bar{\otimes} B(L^2(\G))\bar{\otimes} B(L^2(\G))]\cap [(Y\rtimes_\beta \G)\bar{\otimes} B(L^2(\G))]\cap [(Y\rtimes_\beta \G)\rtimes_{\id \otimes \check{\Delta}_r}\check{\G}]\\
        &=[X\bar{\otimes} B(L^2(\G))\bar{\otimes} B(L^2(\G))]\cap [(Y\rtimes_\beta \G)\rtimes_{\id \otimes \check{\Delta}_r}\check{\G}].
    \end{align*}
    Next, observe that the $\G$-completeness of $(X, \alpha)$ ensures that 
    $$\alpha(X)= (X\bar{\otimes} B(L^2(\G)))\cap \beta(Y).$$
    Therefore, we have
    \begin{align*}
    (X\rtimes_\alpha \G) \rtimes_{\id \otimes \check{\Delta}_r}{\check{\G}}&= [X\bar{\otimes} B(L^2(\G))\bar{\otimes} B(L^2(\G))]\cap [(Y\rtimes_\beta \G)\rtimes_{\id \otimes \check{\Delta}_r}\check{\G}]\\
    &=  [X\bar{\otimes} B(L^2(\G))\bar{\otimes} B(L^2(\G))]\cap [V_{23}^* (\beta(Y)\bar{\otimes} B(L^2(\G)))V_{23}]\\
    &=V_{23}^* [X\bar{\otimes} B(L^2(\G))\bar{\otimes} B(L^2(\G))]\cap [ \beta(Y)\bar{\otimes} B(L^2(\G))]V_{23}\\
    &= V_{23}^*(\alpha(X)\bar{\otimes} B(L^2(\G)))V_{23}\\
    &= V_{23}^*(\alpha\otimes \id)(X \ovot B(L^2(\G)))V_{23}
    \end{align*}
    and the proof is finished. \end{proof}

\begin{Rem} Suppose that $(M, \alpha)$ is a $\G$-dynamical von Neumann algebra and that $X\subseteq M$ is a $\sigma$-weakly closed subspace such that $\alpha(X)\subseteq X \ovot L^\infty(\G)$. In this way, we obtain the $\G$-operator space $(X, \alpha)$. It is an interesting question if the fixed point condition $\alpha(X)= (X\rtimes_\alpha \G)^{\id \otimes \check{\Delta}_r}$ is automatically satisfied in this case. Let us prove that this is indeed the case if $\hat{\G}$ has the approximation property (AP) \cite{DKV24} (in particular, it holds when $\G$ is co-amenable). It follows from \cite[Theorem 4.4]{DKV24} that there exists a net $\{\omega_i\}_{i\in I}\subseteq L^1(\G)$ such that
$(\id \otimes \omega_i)\alpha(m) \to m$ in the $\sigma$-weak topology for every $m\in M$.
Therefore, if $\alpha(m) \in X \ovot L^\infty(\G)$, the fact that $X$ is $\sigma$-weakly closed entails that $m \in X$, proving that 
$(X \ovot L^\infty(\G))\cap \alpha(M)= \alpha(X).$
From this, one easily concludes that the desired condition $(X\rtimes_\alpha \G)^{\id \otimes \check{\Delta}_r}= \alpha(X)$ holds. 
\end{Rem}

Also the following result involving an iterated crossed product will be useful later:
\begin{Prop}\label{cond}
    If $(X, \alpha)$ is a $\G$-operator space, then the map
    $$\Phi: ((X\rtimes_\alpha \G)\ovot L^\infty(\check{\G}), \id \otimes \id \otimes \check{\Delta})\to  ((X\rtimes_\alpha\G)\rtimes_{\id \otimes \Delta_r}\G, \id \otimes \id \otimes \check{\Delta}_r): z \mapsto V_{23}zV_{23}^*$$
    is a $\check{\G}$-equivariant completely isometric isomorphism such that 
    $\Phi \circ (\id \otimes \check{\Delta}_r)=\id \otimes \Delta_l$ on $X\rtimes_\alpha \G$.
\end{Prop}
\begin{proof}
    Let $(Y,\beta)$ be a $\G$-dynamical von Neumann algebra such that $(X, \alpha)\subseteq_\G (Y, \beta)$. Then the map
       $$\Phi: ((Y\rtimes_\beta \G)\ovot L^\infty(\check{\G}), \id \ovot \id \ovot \check{\Delta})\to  ((Y\rtimes_\beta\G)\rtimes_{\id \otimes \Delta_r}\G, \id \ovot \id \ovot \check{\Delta}_r): z \mapsto V_{23}zV_{23}^*$$
       is a $\check{\G}$-equivariant $*$-isomorphism such that $\Phi\circ (\id \otimes \check{\Delta}_r) = \beta \otimes \id$ on $Y \rtimes_\beta \G$. Indeed, that $\Phi$ is well-defined and surjective follows from \cite[Example 1.6]{DD24}, thus $\Phi$ is a $*$-isomorphism. The $\check{\G}$-equivariance 
       $$(\Phi\otimes \id)\circ (\id \otimes \id \otimes \check{\Delta})= (\id \otimes \id \otimes \check{\Delta}_r)\circ \Phi$$
       is easily verified on generators $\beta(y)\otimes 1, 1 \otimes x\otimes 1, 1 \otimes 1 \otimes z$ of $(Y\rtimes_\beta \G)\ovot L^\infty(\check{\G})$, where $y\in Y$ and $x,z \in L^\infty(\check{\G}).$ Also the equality $\Phi\circ (\id \otimes \check{\Delta}_r) = \id \otimes \Delta_l$ is easily verified on generators of $Y\rtimes_\beta \G$. Note then that
       \begin{align*}
           (X\rtimes_\alpha \G)\bar{\otimes} L^\infty(\check{\G})&= [(Y\rtimes_\beta \G)\cap (X\ovot B(L^2(\G)))]\ovot L^\infty(\check{\G})\\
           &= [(Y\rtimes_\beta \G)\ovot L^\infty(\check{\G})]\cap [X\ovot B(L^2(\G))\ovot L^\infty(\check{\G})]\\
           &=  [(Y\rtimes_\beta \G)\ovot L^\infty(\check{\G})]\cap [X\ovot B(L^2(\G))\ovot B(L^2(\G)]
       \end{align*}
       so that 
       \begin{align*}V_{23} [(X\rtimes_\alpha \G)\ovot L^\infty(\check{\G})] V_{23}^*&=  V_{23}[(Y\rtimes_\beta \G)\ovot L^\infty(\check{\G})]\cap [X\ovot B(L^2(\G))\ovot B(L^2(\G)]V_{23}^*\\
       &=V_{23}[(Y\rtimes_\beta \G)\ovot L^\infty(\check{\G})]V_{23}^*\cap V_{23}[X\ovot B(L^2(\G))\ovot B(L^2(\G)]V_{23}^*\\
       &= [(Y\rtimes_\beta \G)\rtimes_{\id \otimes \Delta_r}\G]\cap [X\ovot B(L^2(\G))\ovot B(L^2(\G))]\\
       &= (X\rtimes_\alpha \G)\rtimes_{\id \otimes \Delta_r} \G.
      \end{align*}
The result then easily follows.
\end{proof}

\section{Equivariant injectivity of crossed products}

Once again, fix a locally compact quantum group $\G$. In this section, we prove our main results.

\begin{Theorem}\label{main result}
    Let $(X, \alpha)$ be a $\G$-operator system. The following statements are equivalent:
    \begin{enumerate}
        \item $(X, \alpha)$ is $\G$-injective.
        \item $(X\rtimes_\alpha\G, \id \otimes \check{\Delta}_r)$ is $\check{\G}$-injective and $(X\rtimes_\alpha \G)^{\id \otimes \check{\Delta}_r}= \alpha(X)$.
    \end{enumerate}
\end{Theorem}

\begin{proof}
 $(1)\implies (2)$ Assume $X$ is $\G$-injective. We may assume $X\subseteq B(\mathcal{H})$.
    Since $X$ is $\G$-injective, there is a $\G$-equivariant unital completely positive map
    $$\phi: (B(\mathcal{H})\bar{\otimes}L^\infty(\G), \id \otimes \Delta)\to (X, \alpha)$$
    such that $\phi\circ \alpha = \id_X$. Then applying the crossed product functor $-\rtimes \G$, we obtain a $\check{\G}$-equivariant unital completely positive conditional expectation
    $$\psi: (B(\mathcal{H})\bar{\otimes} B(L^2(\G)), \id \otimes \check{\Delta}_r)\cong ((B(\mathcal{H})\ovot L^\infty(\G))\rtimes_{\id \otimes \Delta}\G, \id \otimes \id \otimes \check{\Delta}_r)\to (X\rtimes_\alpha \G, \id \otimes \check{\Delta}_r).$$
    Therefore, injectivity of $B(\mathcal{H})\bar{\otimes} B(L^2(\G))$ implies injectivity of $X\rtimes_\alpha \G$ as an operator system. It remains to argue that the action $X\rtimes_\alpha \G \curvearrowleft\check{\G}$ is $\check{\G}$-amenable. To see this, note that $\G$-injectivity provides us with a $\G$-equivariant unital completely positive map
    $$E: (X\rtimes_\alpha\G, \id \otimes \Delta_r)\to (X, \alpha), \quad E \circ \alpha = \id_X.$$
    Using Proposition \ref{cond}, we can consider the $\check{\G}$-equivariant unital completely positive map
    $$P: ((X\rtimes_\alpha \G)\ovot L^\infty(\check{\G}), \id \otimes \id \otimes \check{\Delta})\cong ((X\rtimes_\alpha \G)\rtimes_{\id \otimes \Delta_r} \G, \id \otimes \id \otimes \check{\Delta}_r)\stackrel{E\rtimes \G}\longrightarrow (X\rtimes_\alpha \G, \id \otimes \check{\Delta}_r)$$
    that satisfies $P \circ (\id \otimes \check{\Delta}_r) = \id_{X\rtimes_\alpha \G}$. In other words, the action $\id \otimes \check{\Delta}_r: X\rtimes_\alpha \G \curvearrowleft \check{\G}$ is $\check{\G}$-amenable. By Proposition \ref{injective}, $(X\rtimes_\alpha \G, \id \otimes \check{\Delta}_r)$ is $\check{\G}$-injective. By Proposition \ref{completeness}, the statement about the fixed points is clear.

$(2)\implies (1)$ Assume that $(Y, \beta)$ and $(Z, \gamma)$ are $\G$-operator systems, $\iota: Y\to Z$ is a $\G$-equivariant unital complete isometry and $\phi: Y \to X$ is a $\G$-equivariant unital completely positive map. We may assume that there are Hilbert spaces $\mathcal{H}_X$ and $\mathcal{H}_Z$ such that $X\subseteq B(\mathcal{H}_X)$ and $Z \subseteq B(\mathcal{H}_Z)$. Taking crossed products and using that $X\rtimes_\alpha \G$ is $\check{\G}$-injective, we obtain a $\check{\G}$-equivariant unital completely positive map $$\Theta: (Z\rtimes_\gamma \G, \id\otimes \check{\Delta}_r) \to (X\rtimes_\alpha \G, \id \otimes \check{\Delta}_r)$$ such that the diagram on the right commutes:
    $$
\begin{tikzcd}
Y \arrow[d, "\iota"'] \arrow[rr, "\phi"] &  & X &  & Y\rtimes_\beta\G \arrow[d, "\iota \rtimes \G"'] \arrow[rr, "\phi\rtimes \G"] &  & X\rtimes_\alpha \G \\
Z                                        &  &   &  & Z\rtimes_\gamma \G \arrow[rru, "\Theta"', dashed]                              &  &                   
\end{tikzcd}$$
Since $\Theta \circ (\iota \rtimes \G) = \phi\rtimes \G$, we have $\Theta(1\otimes y) = 1 \otimes y$ for all $y\in L^\infty(\check{\G})$. Let $\widetilde{\Theta}: B(\mathcal{H}_Z\otimes L^2(\G)) \to B(\mathcal{H}_X\otimes L^2(\G))$ be a unital completely positive extension of $\Theta$. Then $$V_{23}\in M(1 \otimes C_0(\check{\G})\otimes C_0(\G))\subseteq B(\mathcal{H}_X \otimes L^2(\G)\otimes L^2(\G))$$ is in the multiplicative domain of $\widetilde{\Theta} \otimes \id$ with $(\widetilde{\Theta} \otimes \id)(V_{23})= V_{23}$. Therefore, if $m \in B(\mathcal{H}_Z \otimes L^2(\G))$, 
\begin{align*}
    (\widetilde{\Theta} \otimes \id)(\id \otimes \Delta_r)(m)= (\widetilde{\Theta} \otimes \id)(V_{23}(m\otimes 1)V_{23}^*)=V_{23}(\widetilde{\Theta}(m)\otimes 1)V_{23}= (\id \otimes \Delta_r)\widetilde{\Theta}(m)
\end{align*}
so $\widetilde{\Theta}: (B(\mathcal{H}_Z\otimes L^2(\G)), \id \otimes \Delta_r) \to (B(\mathcal{H}_X\otimes L^2(\G)), \id \otimes \Delta_r)$ is $\G$-equivariant.

Consider the restriction
$$
\begin{tikzcd}
\theta:\gamma(Z) \arrow[rr, "\subseteq"] &  & (Z\rtimes_\gamma \G)^{\id \otimes \check{\Delta}_r} \arrow[rr, "\Theta"] &  & (X\rtimes_\alpha \G)^{\id \otimes \check{\Delta}_r}= \alpha(X)
\end{tikzcd}$$
which is also $\G$-equivariant (since $\widetilde{\Theta}$ is $\G$-equivariant). We define the $\G$-equivariant unital completely positive map
$$\Phi:= \alpha^{-1}\circ\theta\circ\gamma: Z \to X$$
and we note that if $y\in Y$, then
\begin{align*}
    \Phi(\iota(y)) = \alpha^{-1}\theta(\iota \otimes \id)\beta(y) = \alpha^{-1}(\phi\otimes \id)\beta(y) = \alpha^{-1}\alpha(\phi(y)) = \phi(y)
\end{align*}
so that $\Phi\circ \iota = \phi$. 
\end{proof}

\begin{Rem}
    The implication $(1)\implies (2)$ in Theorem \ref{main result} also holds for a $\G$-operator space $(X, \alpha)$. 
    \end{Rem}
Since the fixed-point property is automatically satisfied for $\G$-dynamical von Neumann algebras, we immediately find the following:
    \begin{Cor}\label{main result1}
        If $(M, \alpha)$ is a $\G$-dynamical von Neumann algebra, then $(M, \alpha)$ is $\G$-injective if and only if $(M\rtimes_\alpha \G, \id \otimes \check{\Delta}_r)$ is $\check{\G}$-injective.
    \end{Cor}

Using Proposition \ref{GWamenable}, we then also immediately find the following corollary. This answers a question left open in \cite{DD24}.

\begin{Cor}
    Let $(M, \alpha)$ be a $\G$-dynamical von Neumann algebra. Then $(M, \alpha)$ is $\G$-$W^*$-amenable if and only if $(M\rtimes_\alpha \G, \id \otimes \check{\Delta}_r)$ is $\check{\G}$-$W^*$-amenable.
\end{Cor}

Next, we investigate when $(X\rtimes_\alpha \G, \id \otimes \Delta_r)$ is $\G$-injective. For this, we will use the following easy lemma:

\begin{Lem}[Amplification]\label{amplification}
    If $X$ is an operator system and if $(M, \alpha)$ is a $\G$-dynamical von Neumann algebra, then the following statements are equivalent:
    \begin{enumerate}
        \item $(X \ovot M, \id \otimes \alpha)$ is $\G$-injective.
        \item $X$ is injective and $(M, \alpha)$ is $\G$-injective.
    \end{enumerate}
\end{Lem}
\begin{proof}
    $(1)\implies (2)$ By Proposition \ref{injective}, we see that $X\ovot M$ is injective, whence $X$ is injective. On the other hand, consider a state $f\in X^*$ and consider the unital completely positive map
    $\psi:= f \otimes \id: X \ovot M\to M.$
    Then we have
    \begin{align*}
        (\psi\otimes \id)\circ (\id \otimes \alpha) = (f\otimes \id \otimes \id)\circ (\id \otimes \alpha) = \alpha\circ (f\otimes \id) = \alpha\circ \psi
    \end{align*}
    so $\psi: (X\ovot M, \id\otimes \alpha)\to (M, \alpha)$ is a $\G$-equivariant unital completely positive conditional expectation onto $1\otimes M \cong M$. Therefore, $(M, \alpha)$ is $\G$-injective as well.

    $(2)\implies (1)$ By Proposition \ref{injective}, it is clear that $X\ovot M$ is injective. Since $(M, \alpha)$ is $\G$-injective, there is a $\G$-equivariant unital completely positive map
    $E: (M\ovot L^\infty(\G), \id \otimes \Delta)\to (M, \alpha)$ such that $E\circ \alpha = \id_M$.
    Then 
    $$\id \otimes E: (X \ovot M \ovot L^\infty(\G), \id \otimes \id \otimes \Delta)\to (X \ovot M, \id \otimes \alpha)$$
    is a $\G$-equivariant unital completely positive map such that $(\id \otimes E)\circ (\id \otimes \alpha)= \id_{X\ovot M}$. Therefore, we see that $(X\ovot M, \id \otimes \alpha)$ is $\G$-amenable. By Proposition \ref{injective}, we conclude that $(X\ovot M, \id \otimes \alpha)$ is $\G$-injective.
\end{proof}

\begin{Theorem}\label{mainresult2}
    Let $(X, \alpha)$ be a $\G$-operator system. The following statements are equivalent:
    \begin{enumerate}
        \item $(X\rtimes_\alpha\G, \id \otimes \Delta_r)$ is $\G$-injective.
        \item $X\rtimes_\alpha\G$ is injective as an operator system and $\G$ is amenable. 
    \end{enumerate}
\end{Theorem}
\begin{proof} We start by noting that 
$$((X\rtimes_\alpha \G)\ovot L^\infty(\check{\G}))^{\id \otimes \id \otimes \check{\Delta}}= (X\rtimes_\alpha \G)\ovot \mathbb{C}1$$
so by Proposition \ref{cond} we find
$$((X\rtimes_\alpha \G)\rtimes_{\id \otimes \Delta_r}\G)^{\id \otimes \id \otimes \check{\Delta}_r}= V_{23}((X\rtimes_\alpha \G)\ovot \mathbb{C}1)V_{23}^*= (\id \otimes \Delta_r)(X\rtimes_\alpha \G).$$

Therefore, from Theorem \ref{main result} we know that $(X\rtimes_\alpha \G, \id \otimes \Delta_r)$ is $\G$-injective if and only if
$$((X\rtimes_\alpha \G)\rtimes_{\id \otimes \Delta_r}\G, \id \otimes \id \otimes \check{\Delta}_r)\cong ((X\rtimes_\alpha \G)\ovot L^\infty(\check{\G}), \id \otimes \id \otimes \check{\Delta})$$
is $\check{\G}$-injective. By Lemma \ref{amplification}, this is equivalent with the statement that $X\rtimes_\alpha \G$ is injective and $(L^\infty(\check{\G}), \check{\Delta})$ is $\check{\G}$-injective. But by combining Corollary \ref{main result1} and Corollary \ref{amenab}, we know that $\G$ is amenable if and only if $(L^\infty(\check{\G}), \check{\Delta})$ is $\check{\G}$-injective.
\end{proof}

\section{Amenability and inner amenability of actions}

Fix a locally compact quantum group $\G$. As we have seen, the notion of equivariant injectivity is closely related to the notion of amenability of an action. Hence, in this section, we investigate the notion of $\G$-amenability, and the related notion of inner $\G$-amenability, in some more detail. 

\subsection{Inner actions} We start with a study of inner actions. Let $(M, \alpha)$ be a $\G$-dynamical von Neumann algebra. The action $\alpha: M \curvearrowleft \G$ is called \emph{inner} if there is a unitary $U \in M \ovot L^\infty(\G)$ satisfying $(\id \otimes \Delta)(U) = U_{12}U_{13}$ such that $\alpha(m) = U(m\otimes 1)U^*$ for all $m\in M$.

The following lemma shows that the crossed product with respect to an inner action is very easy to understand:

\begin{Lem}\label{wellwell}
    Assume that $M$ is a von Neumann algebra, and $U\in M \ovot L^\infty(\G)$ a unitary such that $(\id \otimes \Delta)(U) = U_{12}U_{13}$. Consider the $\G$-action
    $$\alpha: M \to M \ovot L^\infty(\G): m \mapsto U(m\otimes 1)U^*.$$
    Then 
    $M\rtimes_\alpha \G= U(M \ovot L^\infty(\check{\G}))U^*.$ Consequently, the map $$\psi: (M\rtimes_\alpha \G, \id \otimes \check{\Delta}_r) \to (M \ovot L^\infty(\check{\G}), \id \otimes \check{\Delta}): z \mapsto U^*z U$$ is a $\check{\G}$-equivariant isomorphism.
\end{Lem}
\begin{proof} If $x\in L^\infty(\check{\G})$, we have
\begin{align*}
    (\alpha \otimes \id)(U(1\otimes x)U^*) &= U_{12}U_{13}U_{12}^* (1\otimes 1 \otimes x) U_{12}U_{13}^* U_{12}^*\\
    &= (\id \otimes \Delta)(U)(1\otimes 1 \otimes x) (\id \otimes \Delta)(U)^*\\
    &= (\id \otimes \Delta)(U) (\id \otimes \Delta_l)(1\otimes x) (\id \otimes \Delta)(U^*)\\
    &= (\id \otimes \Delta_l)(U(1\otimes x)U^*)
\end{align*}
so that $U(1\otimes x)U^*\in M\rtimes_\alpha \G$. We conclude that
$$U(M \ovot L^\infty(\check{\G}))U^*\subseteq M\rtimes_\alpha \G.$$
The other inclusion follows immediately from Lemma \ref{well-defined}.   The $\check{\G}$-equivariance of $\psi$ is immediate from the fact that $\check{V}_{23}$ and 
$U_{12}$ commute.
\end{proof}

\begin{Prop}
    Let $(M, \alpha)$ be a $\G$-dynamical von Neumann algebra where $\alpha$ is an inner action. Then:
    \begin{enumerate}
        \item $(M, \alpha)$ is $\G$-amenable if and only if $\G$ is amenable.
        \item $(M, \alpha)$ is inner $\G$-amenable if and only if $\G$ is inner amenable.
        \item $(M, \alpha)$ is $\G$-injective if and only if $\G$ is amenable and $M$ is injective.
    \end{enumerate}
\end{Prop}
\begin{proof} 
$(1)$ By \cite[Theorem 6.12]{DD24} and Lemma \ref{wellwell}, we see that $(M, \alpha)$ is $\G$-amenable if and only if $(M \ovot L^\infty(\check{\G}), \id \otimes \check{\Delta})$ is $\check{\G}$-inner amenable. This happens if and only if $(L^\infty(\check{\G}),\check{\Delta})\cong_{\check{\G}} (\mathbb{C}\rtimes_\tau \G, \id \otimes \check{\Delta}_r)$ is inner-$\check{\G}$-amenable, which is equivalent with amenability of $\G$ by another application of \cite[Theorem 6.12]{DD24}.

$(2)$ Similar to $(1)$.

    $(3)$ This follows immediately by combining $(1)$ and Proposition \ref{injective}.
\end{proof}

\begin{Cor}\label{innercor}
    Let $(M, \alpha)$ be a $\G$-dynamical von Neumann algebra. Then:
    \begin{enumerate}
        \item $(M\rtimes_\alpha \G, \id \otimes \Delta_r)$ is $\G$-amenable if and only if $\G$ is amenable.
        \item $(M\rtimes_\alpha \G, \id \otimes \Delta_r)$ is inner $\G$-amenable if and only if $\G$ is inner amenable.
    \end{enumerate}
\end{Cor}

Note that this complements \cite[Theorem 6.12]{DD24}. 

\subsection{Dynamical characterisation of (inner) amenability of a locally compact quantum group}

We are now in a position to give an easy and conceptual proof of the following dynamical characterisation of amenability and inner amenability of locally compact quantum groups:

\begin{Theorem}\label{amenable}   The following statements are equivalent:
\begin{itemize}
    \item[(a)] $\G$ is inner amenable, i.e. the trivial action $\tau: \mathbb{C}\curvearrowleft\G$ is inner amenable.
    \item[(b)] The action $\check{\Delta}: L^\infty(\check{\G})\curvearrowleft \check{\G}$ is amenable.
    \item[(c)] The action $\Delta_r: L^\infty(\check{\G})\curvearrowleft \G$ is inner amenable.
    \item[(d)] The action $\Delta_r: B(L^2(\G))\curvearrowleft \G$ is inner amenable.
\end{itemize} Moreover, all the following statements are equivalent:
    \begin{enumerate}
        \item $\G$ is amenable, i.e. the trivial action $\mathbb{C}\curvearrowleft \mathbb{G}$ is amenable.
          \item The action $\Delta_r: L^\infty(\check{\G})\curvearrowleft \G$ is amenable.
        \item The action $\Delta_r: B(L^2(\G))\curvearrowleft \G$ is amenable.
        \item The action $\check{\Delta}: L^\infty(\check{\G})\curvearrowleft \check{\G}$ is inner amenable.
        \item $(\mathbb{C}, \tau)$ is $\G$-injective.
         \item $(L^\infty(\check{\G}), \check{\Delta})$ is $\check{\G}$-injective.
         \item $(B(L^2(\G)), \Delta_r)$ is $\G$-injective.
        \item $L^\infty(\check{\G})$ is injective and $\G$ is inner amenable.
        \item There is a $\check{\G}$-equivariant unital completely positive conditional expectation
        $$Q: (B(L^2(\G)), \check{\Delta}_r)\to (L^\infty(\check{\G}),\check{\Delta}).$$
    \end{enumerate}
\end{Theorem}
\begin{proof}  We recall that there are equivariant isomorphisms
\begin{align}
    \label{1}(\mathbb{C}\rtimes_\tau \G, \id \otimes \Delta_r)&\cong_\G (L^\infty(\check{\G}),\Delta_r)\\
\label{2}(\mathbb{C}\rtimes_\tau \G, \id \otimes \check{\Delta}_r)&\cong_{\check{\G}}(L^\infty(\check{\G}), \check{\Delta}) \\
   \label{3}(L^\infty(\G)\rtimes_\Delta \G, \id \otimes \Delta_r) &\cong_{\G}  (B(L^2(\G)), \Delta_r)\\
   \label{4}(L^\infty(\G)\rtimes_\Delta \G, \id \otimes \check{\Delta}_r) &\cong_{\check{\G}} (B(L^2(\G)), \check{\Delta}_r)
\end{align}

The equivalence $(a)\iff (b)$ follows from \cite[Theorem 6.12]{DD24} and \eqref{2}. The equivalence $(a)\iff (c)$ follows from \eqref{1} and Corollary \ref{innercor}. The equivalence $(a)\iff (d)$ follows from \eqref{3} and Corollary \ref{innercor}.

The equivalence $(1)\iff (2)$ follows from Corollary \ref{innercor} and \eqref{1}. The equivalence $(1)\iff (3)$ follows from Corollary \ref{innercor} and \eqref{3}. The equivalence $(1)\iff (4)$ follows from \cite[Theorem 6.12]{DD24} and \eqref{2}. The equivalence $(1)\iff (5)$ is Corollary \ref{amenab}. The equivalence $(5)\iff (6)$ follows from \eqref{2} and Corollary \ref{main result1}. The equivalence $(3)\iff (7)$ follows from Proposition \ref{injective}. The equivalence $(6)\iff (8)$ follows from Proposition \ref{injective} and the equivalence $(a)\iff (b)$. The equivalence $(1)\iff (9)$ follows from Proposition \ref{general result}, \eqref{2} and \eqref{4}.
\end{proof}
\begin{Rem}
    Some of the results in Theorem \ref{amenable} are already known. For example:
    \begin{itemize}
        \item The equivalence $(1) \iff (6)$ is \cite[Theorem 5.2]{Cr17a}.
        \item The equivalence $(1) \iff (9)$ was already established in \cite{CN16}.
        \item The equivalence $(1) \iff (8)$ is \cite[Corollary 3.9]{Cr19}.
        \item Note also the result \cite[Theorem 5.5]{CN16}, which tells us that $\G$ is amenable if and only if $B(L^2(\G))$ is injective as a (left) operator $B(L^2(\G))_*$-module where the module action is given by
        $$\omega \RHD x = (\id \otimes \omega)\Delta_r(x), \quad x \in B(L^2(\G)), \omega \in B(L^2(\G))_*.$$ 
    The equivalence $(1) \iff (7)$ and Proposition \ref{injective} tell us that $\G$ is amenable if and only if $B(L^2(\G))$ is injective as an operator $L^1(\G)$-module. 
    In fact, if $(X, \rhd)$ is a left $L^1(\G)$-module, then it is also a left $B(L^2(\G))_*$-module for 
    the action $$\omega \RHD x:= \omega\vert_{L^\infty(\G)}\rhd x, \quad \omega \in B(L^2(\G))_*, x \in X.$$
    It is then easy to see that if $(X, \RHD)$ is injective as a $B(L^2(\G))_*$-module, then $(X, \rhd)$ is also injective as an $L^1(\G)$-module. However, the converse is not clear in general. In the case of $B(L^2(\G))$, we thus see that both notions of module-injectivity agree.
    \end{itemize}
\end{Rem}

\begin{Rem} The equivalence $(a)\iff (b)$ in Theorem \ref{amenable} can be used to construct examples of amenable locally compact quantum groups that have non-amenable actions. For example, let $G$ be a locally compact group that is not inner amenable (see e.g. \cite[Remark 2.6 (ii)]{FSW17} together with \cite[Proposition 3.2]{CZ17}). Then $\check{G}$ is an amenable locally compact quantum group with function algebra the right group von Neumann algebra $\mathscr{R}(G)$ and coproduct uniquely determined by $\check{\Delta}(\rho_g) = \rho_g\otimes \rho_g$ for $g\in G$. The action $\check{\Delta}: \mathscr{R}(G)\curvearrowleft\check{G}$ is not amenable.
\end{Rem} 

\subsection{Application to the non-commutative Poisson boundary}

We give an application of the above results to the non-commutative Poisson boundary \cite{Izu02}. We briefly recall the details of the construction. Let $(M, \alpha)$ be a $\G$-dynamical von Neumann algebra and consider a normal $\G$-equivariant unital completely positive map $P: (M, \alpha) \to (M, \alpha)$. We then define the operator system of \emph{$P$-harmonic elements}
$$\mathcal{H}^\infty(M,P) = \{x\in M: P(x)= x\}.$$
It is then easily checked that $\alpha(\mathcal{H}^\infty(M,P))\subseteq \mathcal{H}^\infty(M,P)\ovot L^\infty(\G)$. Then the restriction $\alpha_P: \mathcal{H}^\infty(M,P)\to \mathcal{H}^\infty(M,P)\ovot L^\infty(\G)$ turns $\mathcal{H}^\infty(M,P)$ into a $\G$-operator system. Taking a cluster point of the sequence
$$\left\{\frac{1}{n}\sum_{k=1}^n P^k\right\}_{n=1}^\infty$$
of unital completely positive maps $M\to M$ in the point $\sigma$-weak topology, we find a $\G$-equivariant unital completely positive conditional expectation
$$E: M\to  \mathcal{H}^\infty(M,P).$$
Of particular importance is the case where $(M, \alpha)= (L^\infty(\G), \Delta)$. In that case, to any state $\mu \in C_u(\G)^*$ \cite{Kus01}, we can associate a \emph{Markov operator} $P_\mu: (L^\infty(\G), \Delta)\to (L^\infty(\G), \Delta)$ \cite[Section 2]{KNR13}, which is a $\G$-equivariant, normal, unital completely positive map. In that case, we write $\mathcal{H}_\mu := \mathcal{H}^\infty(L^\infty(\G), P_\mu)$ and $\Delta_\mu:=\Delta_{P_\mu}$. Note that with $\mu = \epsilon \in C_u(\G)^*$ the counit, we have $\mathcal{H}_\mu = L^\infty(\G)$ and $\Delta_\mu = \Delta$.

From Theorem \ref{amenable} and the existence of the map $E$, the following result is then clear:

\begin{Prop}\label{Poisson}
  Let $(M,\alpha)$ be a $\G$-dynamical von Neumann algebra. 
  \begin{enumerate}
      \item If $\alpha: M\curvearrowleft \G$ is an amenable action, then so is $\alpha_P: \mathcal{H}^\infty(M,P)\curvearrowleft \G$.
      \item If $\alpha: M\curvearrowleft \G$ is an inner amenable action, then so is $\alpha_P: \mathcal{H}^\infty(M,P)\curvearrowleft \G$.
      \item If $\alpha: M \curvearrowleft \G$ is $\G$-injective, then so is $\alpha_P: \mathcal{H}^\infty(M,P)\curvearrowleft \G$. Consequently, $\mathcal{H}^\infty(M,P)\rtimes_{\alpha_P} \G$ is $\check{\G}$-injective.
  \end{enumerate}
  In particular, the following statements are equivalent:
    \begin{enumerate}
        \item[(a)] $\check{\G}$ is inner amenable if and only if $\Delta_\mu: \mathcal{H}_\mu\curvearrowleft\G$ is amenable for all $\mu \in C_u(\G)^*$.
        \item[(b)] $\check{\G}$ is amenable if and only if $\Delta_\mu: \mathcal{H}_\mu\curvearrowleft\G$ is $\G$-injective for all $\mu \in C_u(\G)^*$.
    \end{enumerate}
    In particular, if $\check{\G}$ is amenable, the crossed product $\mathcal{H}_\mu \rtimes_{\Delta_\mu} \G$ is $\check{\G}$-injective for all $\mu \in C_u(\G)^*.$
\end{Prop}

\begin{Rem}
    If $\G$ is a discrete quantum group, Proposition \ref{Poisson} was already known, and follows e.g.\ from the results in \cite{Moa18} or \cite{DH24}. 
\end{Rem}

\subsection{Automatic (inner) amenability of actions}

In this subsection, we investigate some situations where (inner) amenability of an action comes for free.

\begin{Prop}\label{counitslice}
    Let $(M, \alpha)$ be a $\G$-dynamical von Neumann algebra.
    \begin{enumerate}
        \item If $\G$ is co-amenable and $M\rtimes_\alpha\G$ is injective, then so is $M$.
        \item If $\G$ is co-amenable and if $(M, \alpha)$ is $\G$-amenable, then it is also inner $\G$-amenable.
        \item If $\check{\G}$ is co-amenable and if $(M, \alpha)$ is inner $\G$-amenable, then it is also $\G$-amenable.
    \end{enumerate}
\end{Prop}
\begin{proof} Let $\epsilon: C_0(\G)\to \mathbb{C}$ be the counit and $\tilde{\epsilon}\in B(L^2(\G))^*$ a state extending $\epsilon$. Since $(\id \otimes \Delta)(V)= V_{12}V_{13}$, we have $(\id \otimes \tilde{\epsilon})(V)= 1$. If $y\in L^\infty(\G)$, we have
    $$(\id \otimes \tilde{\epsilon})\Delta(y) = (\id \otimes \tilde{\epsilon})(V(y\otimes 1)V^*) = y$$
    by a standard multiplicative domain argument.
    Consider the map
    $$E: = \id \otimes \tilde{\epsilon}: M \rtimes_\alpha \G \to M.$$
    Then we have for $x\in M$,
    \begin{align*}
        \alpha (E(\alpha(x))) = (\id \otimes \id \otimes \tilde{\epsilon})(\alpha\otimes \id)\alpha(x)= (\id \otimes \id \otimes \tilde{\epsilon})(\id \otimes \Delta)\alpha(x)= \alpha(x)
    \end{align*}
   so that by injectivity of $\alpha$ we find $E\circ \alpha = \id_M$.
   
   $(1)$ The injectivity of $M\cong \alpha(M)$ then follows immediately from the injectivity of $M\rtimes_\alpha \G$.

   $(2)$ Let $F: M \ovot L^\infty(\G)\to M$ be a $\G$-equivariant map satisfying $F\circ \alpha = \id_M$. Define then
$$\widetilde{E}:=F \circ (E\otimes \id)\circ (\id \otimes \Delta_r): M\rtimes_\alpha \G \to M$$
which is $\G$-equivariant. Moreover, it satisfies
\begin{align*}
    \widetilde{E}(\alpha(x)) =  F(E\otimes \id)(\alpha\otimes \id)\alpha(x) = F(\alpha(x)) = x
\end{align*}
for all $x\in M$, so $\alpha: M\curvearrowleft \G$ is inner $\G$-amenable.

$(3)$ This follows immediately from $(2)$ and \cite[Theorem 6.12]{DD24}.
\end{proof}

\begin{Rem}
    \begin{itemize}
        \item $B(L^2(\G))= L^\infty(\G)\rtimes \G$ is always injective, yet $L^\infty(\G)$ may not be injective when $\G$ is not co-amenable. Thus, $(1)$ is not true in general.
        \item Let $\G$ be a locally compact quantum group that is inner amenable, but not amenable. Then the action $\check{\Delta}: L^\infty(\check{\G})\curvearrowleft \check{\G}$ is $\check{\G}$-amenable, but not inner $\check{\G}$-amenable. Thus, also $(2)$ is not true in general.
        \item If $\G$ is a locally compact group or a discrete quantum group, every $\G$-dynamical von Neumann algebra $(M, \alpha)$ that is $\G$-amenable is automatically also inner $\G$-amenable. In particular, $(2)$ in Proposition \ref{counitslice} removes the injectivity assumption in \cite[Proposition 4.4]{MP22} for classical locally compact groups.
    \end{itemize}
\end{Rem}

\begin{Prop}
    Let $\G$ be a unimodular discrete quantum group. Then every $\G$-dynamical von Neumann algebra is inner $\G$-amenable. If moreover $\G$ is amenable, then every $\G$-dynamical von Neumann algebra is also $\G$-amenable.
\end{Prop}
\begin{proof} Let $\check{\Phi}$ be the Haar state of $\check{\G}$ and define $\omega := \omega_{\Lambda_{\check{\Phi}}(1)} \in B(\ell^2(\G))_*$. 
Since $\G$ is unimodular, it is well-known that
$$(\omega \otimes \id)\circ \Delta_r = (\id \otimes \omega)\circ \Delta_l$$
on $B(\ell^2(\G))$. Let $(M, \alpha)$ be a $\G$-dynamical von Neumann algebra. Define 
    $$E: M\rtimes_\alpha \G \to M: z \mapsto (\id \otimes \omega)(z)$$
    and note that for $z\in M\rtimes_\alpha \G$, we have
    \begin{align*}
        (E\otimes \id)( \id \otimes \Delta_r)(z) &= (\id \otimes \omega \otimes \id)(\id \otimes \Delta_r)(z) \\
        &= (\id \otimes \id \otimes \omega)(\id \otimes \Delta_l)(z)\\
        &= (\id \otimes \id \otimes \omega)(\alpha \otimes \id)(z)= \alpha(E(z))
    \end{align*}
    so the map $E$ is $\G$-equivariant. Note also that $E\circ \alpha = \id_M$ since $\epsilon = \omega\vert_{\ell^\infty(\G)}$ is the counit of $\ell^\infty(\G)$. This shows that $(M, \alpha)$ is inner $\G$-amenable. 

    If moreover $\G$ is amenable, then $\check{\G}$ is co-amenable, and the result follows from Proposition \ref{counitslice}.
\end{proof}
Using \cite[Theorem 6.12]{DD24}, we then immediately find:
\begin{Cor}\label{coro}
   Let $\G$ be a compact quantum group of Kac type. Then every $\G$-dynamical von Neumann algebra is $\G$-amenable.
\end{Cor}

\begin{Rem}\label{MoaRem}
    In \cite[Theorem 4.7]{Moa18}, it is claimed that if $\G$ is an amenable discrete quantum group, that every $\G$-dynamical von Neumann algebra is automatically $\G$-amenable. However, the proof of \cite[Theorem 4.7]{Moa18} contains a mistake, as it implicitly uses unimodularity of $\G$, as observed in \cite{AK24}. Perhaps surprisingly, it turns out that there exist amenable discrete quantum groups acting non-amenably on a von Neumann algebra. Examples of this will be given in a forthcoming joint work.
\end{Rem}

The following proposition is an equivariant version of \cite[Corollaire 4.3]{AD79}.
\begin{Prop}\label{ADAPPL}
    Let $G$ be a locally compact group and $(M, \alpha)$ a $G$-dynamical von Neumann algebra where $M$ is a factor. The following statements are equivalent:
    \begin{enumerate}
        \item $(M, \alpha)$ is $G$-injective.
        \item $M$ is injective and $G$ is amenable.
    \end{enumerate}
\end{Prop}
\begin{proof} This follows immediately from Proposition \ref{injective}, Theorem \ref{main result} and \cite[Remarques 3.5(c)]{AD79}.
\end{proof}
\section{Equivariant injective envelopes of crossed products}

In this section, we examine the question if the injective envelope construction is compatible with the crossed product functor. More specifically, in this section we prove that given a $\G$-operator system $(X, \alpha)$, we have the identification
$$I_{\check{\G}}^1(X\rtimes_\alpha \G)= I_\G^1(X)\rtimes \G$$
of $\check{\G}$-operator systems, where $\G$ is either a compact or a discrete quantum group. We were not able to prove versions of this result for other classes of locally compact quantum groups. Even for arbitrary classical locally compact groups, this remains open.

\begin{Rem}
    The $\check{\G}$-operator systems $I_{\check{\G}}^1(X\rtimes_\alpha \G)$ and $I_\G^1(X)\rtimes \G$ are both injective as operator systems, so they become unital $C^*$-algebras via the Choi-Effros product. Then the identification  $I_{\check{\G}}^1(X\rtimes_\alpha \G)= I_\G^1(X)\rtimes \G$ becomes an identification of unital $C^*$-algebras.
\end{Rem}

\subsection{Discrete quantum groups} In this subsection, let $\G$ be a discrete quantum group.

\begin{Lem}\label{rigidlemma} Let $(X, \alpha)$ a $\G$-operator system. If $T: X\rtimes_\alpha\G \to X\rtimes_\alpha \G$ is a $\check{\G}$-equivariant unital completely positive map such that 
$T(\alpha(x)(1\otimes y)) = \alpha(x)(1\otimes y)$
for all $x\in X$ and $y\in L^\infty(\check{\G})$, then $T= \id_{X\rtimes_\alpha \G}$.
\end{Lem}
\begin{proof}
    Consider the completely isometric isomorphism $\Phi: X \ovot B(L^2(\G))\to (X\rtimes_\alpha \G)\rtimes_{\id \otimes \check{\Delta}_r} \check{\G}$ from Proposition \ref{completeness} (3). Define 
    $$S: X \bar{\otimes} B(L^2(\G)) \stackrel{\Phi}\cong (X\rtimes_\alpha\G)\rtimes_{\id \otimes \check{\Delta}_r} \check{\G}\stackrel{T\rtimes \check{\G}}\longrightarrow (X\rtimes_\alpha \G)\rtimes_{\id \otimes \check{\Delta}_r} \check{\G}\stackrel{\Phi^{-1}}\cong X \bar{\otimes} B(L^2(\G)).$$
    Then we claim that 
    $S(\alpha(x)) = \alpha(x)$ and $S(1\otimes y) = 1\otimes y$
    for all $x\in X$ and all $y\in B(L^2(\G))$. Indeed, fix $y\in B(L^2(\G))$. Then $$\Phi(1\otimes y)= V_{23}^*(\alpha \otimes \id)(1\otimes y)V_{23}= 1 \otimes V^*(1\otimes y)V.$$
     Since $V \in M(C(\check{\G})\otimes c_0(\G))$, a standard multiplicative domain argument shows that
    $$(T\rtimes \check{\G})(1\otimes V^*(1\otimes y)V)   = 1\otimes V^*(1\otimes y)V$$
    so it is clear that $S(1\otimes y) = 1 \otimes y$. Similarly, if $x\in X$, we have $$\Phi(\alpha(x))= V_{23}^*(\alpha\otimes \id)\alpha(x)V_{23}=V_{23}^*(\id \otimes \Delta)\alpha(x)V_{23}= \alpha(x)\otimes 1$$
    so that 
    $$S(\alpha(x)) = \Phi^{-1}((T\rtimes \check{\G})(\alpha(x)\otimes 1))=\Phi^{-1}(\alpha(x)\otimes 1)= \alpha(x).$$

    Let $\{e_i\}_{i\in I}$ be an orthonormal basis for $L^2(\G)$ and $\{E_{s,t}\}_{s,t\in I}\subseteq B(L^2(\G))$ be the corresponding matrix units. Since the counit $\epsilon \in \ell^1(\G)$ is a vector state, we may assume that there is $k\in I$ such that $\epsilon = \omega_{e_{k}, e_{k}}$.
    
    Then we see that if $x\in X$ and $i,j\in I$ that
    $$(1\otimes E_{i,k})\alpha(x)(1\otimes E_{k,j}) = (\id \otimes \omega_{e_k,e_k})(\alpha(x))\otimes E_{i,j} = (\id \otimes \epsilon)\alpha(x)\otimes E_{i,j}= x\otimes E_{i,j}.$$
    Therefore, by a multiplicative domain argument, for $x\in X$ and $i,j\in I$, we find
    $$S(x\otimes E_{i,j}) = S((1\otimes E_{i,k})\alpha(x)(1\otimes E_{k,j})) = (1\otimes E_{i,k})\alpha(x)(1\otimes E_{k,j}) = x\otimes E_{i,j}.$$

A standard multiplicative domain argument implies that $S= \id_{X \bar{\otimes} B(L^2(\G))}$. It follows that also $T$ is the identity map.
\end{proof}

\begin{Prop}\label{rigid}
    The following statements are equivalent for $\G$-operator systems $(X, \alpha)$ and $(Y, \beta)$:
    \begin{enumerate}
        \item $(Y, \iota: X \to Y)$ is a $\G$-rigid extension of $X$.
        \item $(Y\rtimes_\beta \G, \iota \rtimes \G: X\rtimes_\alpha \G \to Y\rtimes_\beta \G)$ is a $\check{\G}$-rigid extension of $X\rtimes_\alpha \G$.
    \end{enumerate}
\end{Prop}

\begin{proof}
    $(1)\implies (2)$ Let $\psi: Y \rtimes_\beta \G \to Y\rtimes_\beta \G$ be a unital completely positive $\check{\G}$-equivariant map such that $\psi\circ (\iota \rtimes \G) = \iota\rtimes \G$. In particular, we know that 
    $\psi(1\otimes y) = 1 \otimes y$ for all $y\in L^\infty(\check{\G})$,
    from which it follows that the restriction
    $$\phi: Y\cong \beta(Y)\to \beta(Y)\cong Y$$
    is $\G$-equivariant and satisfies $\phi\circ \iota = \iota$. The $\G$-rigidity then implies that $\phi = \id_Y$. Therefore, it is clear that
    $$\psi(\beta(y)(1\otimes z)) = \beta(y)(1\otimes z), \quad y \in Y, z \in L^\infty(\check{\G})$$
    so $\psi= \id$ by Lemma \ref{rigidlemma}.

    $(2)\implies (1)$ Suppose that $\phi: Y \to Y$ is a unital completely positive $\G$-equivariant map such that $\phi\circ \iota = \iota$. Then by functoriality also 
    $(\phi\rtimes \G)(\iota \rtimes \G)= \iota \rtimes \G$
    so that by $\check{\G}$-rigidity we have $\phi\rtimes \G = \id.$ Therefore,
    $\beta(y) = (\phi\rtimes \G)\beta(y) = \beta(\phi(y))$ for $y\in Y$, whence $y = \phi(y)$. 
\end{proof}

\begin{Prop}\label{injective envelope crossed product}
     The following statements are equivalent for $\G$-operator systems $(X, \alpha)$ and $(Y, \beta)$:
    \begin{enumerate}
        \item $(Y, \iota: X \to Y)$ is a $\G$-injective envelope of $X$.
        \item $(Y\rtimes_\beta \G, \iota\rtimes \G: X\rtimes_\alpha \G \to Y \rtimes_\beta \G)$ is a $\check{\G}$-injective envelope of $X\rtimes_\alpha \G$.
    \end{enumerate}
    In particular,
    $I_{\check{\G}}^1(X\rtimes_\alpha \G)= I_\G^1(X)\rtimes \G$
    as $\check{\G}$-operator systems.
\end{Prop}
\begin{proof}
    This follows immediately by combining Proposition \ref{injective envelopes}, Theorem \ref{main result}, Proposition \ref{rigid} and the fact that if $\G$ is a discrete quantum group, then automatically $(Y\rtimes_\beta \G)^{\id \otimes \check{\Delta}_r}= \beta(Y)$ (see the remark following Definition \ref{com} and Proposition \ref{completeness}).
\end{proof}

\begin{Rem}
    Using the terminology of \cite{DH24}, Proposition \ref{injective envelope crossed product} can be seen as the analoguous result of \cite[Theorem 5.5]{DH24} but for $\G$-$W^*$-operator systems instead of $\G$-$C^*$-operator systems.
\end{Rem}

\begin{Cor}
    $I_{\check{\G}}^1(L^\infty(\check{\G}))= I_\G^1(\mathbb{C})\rtimes \G = C(\partial_F \G)\rtimes \G$ where $\partial_F \G$ is the non-commutative Furstenberg boundary \cite{KKSV22}.
\end{Cor}

\subsection{Compact quantum groups} In this subsection, let $\G$ be a compact quantum group with Haar state $\Phi$ and associated GNS-map $\Lambda_\Phi: L^\infty(\G)\to L^2(\G)$. A unitary finite-dimensional representation $\pi$ of $\G$ consists of a pair $(\mathcal{H}_\pi, U_\pi)$ where $\mathcal{H}_\pi$ is a finite-dimensional Hilbert space and $U_\pi \in B(\mathcal{H}_\pi)\otimes C(\G)$ a unitary such that $(\id \otimes \Delta)(U_\pi) = U_{\pi, 12}U_{\pi, 13}$ in $B(\mathcal{H}_\pi)\otimes C(\G)\otimes C(\G)$. Given $\xi, \eta \in \mathcal{H}_\pi$, we define the matrix coefficient
$U_\pi(\xi, \eta):= (\omega_{\xi, \eta}\otimes \id)(U_\pi)$ and we will write $n_\pi:= \dim(\mathcal{H}_\pi)$. If $\{e_i^\pi\}_{i=1}^{n_\pi}$ is an orthonormal basis for $\mathcal{H}_\pi$, then 
$$\Delta(U_\pi(\xi, \eta)) = \sum_{i=1}^{n_\pi} U_\pi(\xi, e_i^\pi)\otimes U_\pi(e_i^\pi, \eta).$$ 
If $\pi, \chi$ are unitary finite-dimensional representations, then a linear map $T: \mathcal{H}_\pi \to \mathcal{H}_\chi$ is called intertwiner if $(T\otimes 1)U_\pi = U_\chi(T\otimes 1)$ and we write $\Mor(\pi, \chi)$ for the set of such intertwiners. The representations $\pi, \chi$ are called (unitarily) equivalent if there exists a unitary operator in $\Mor(\pi, \chi)$. A unitary finite-dimensional representation $\pi$ of $\G$ is called irreducible if $\Mor(\pi, \pi)$ is one-dimensional. We will fix a maximal family $\Irr(\G)$ of pairwise non-equivalent irreducible representations.
The space linearly spanned by all matrix coefficients is denoted by $\mathcal{O}(\G)$. In fact, the elements 
$\{U_\pi(e_i^\pi, e_j^\pi): \pi\in \Irr(\G), 1\le i,j\le n_\pi\}$ form a Hamel basis for $\mathcal{O}(\G)$. The space $\mathcal{O}(\G)$ is norm-dense in $C(\G)$ and $(\mathcal{O}(\G), \Delta)$ forms a Hopf $^*$-algebra. If $\pi\in \Irr(\G)$, there exists a canonical positive invertible operator $Q_\pi\in B(\mathcal{H}_\pi)$ such that $\Tr(Q_\pi^{1/2}) = \Tr(Q_\pi^{-1/2})$. We write $\dim_q(\pi)$ for the latter value and the following orthogonality relations hold:
\begin{align*}
    &\Phi(U_\pi(\xi, \eta)U_\chi(\xi', \eta')^*) = \delta_{\pi, \chi}\frac{\langle \xi, \xi'\rangle \langle\eta', Q_\pi^{-1/2}\eta\rangle}{\dim_q(\pi)},\\
    &\Phi(U_\pi(\xi, \eta)^*U_\chi(\xi', \eta')) = \delta_{\pi, \chi}\frac{\langle \xi', Q_\pi^{1/2}\xi\rangle \langle\eta, \eta'\rangle}{\dim_q(\pi)}.
\end{align*}

If $\pi\in \Irr(\G)$, define $\chi_\pi:= \dim_q(\pi) \sum_{i=1}^{n_\pi} U_\pi(e_i^\pi, Q_\pi^{1/2}e_i^\pi)\in \mathcal{O}(\G)$. We write 
$$c_c(\tilde{\G}) = \{\Phi(-a): a \in \mathcal{O}(\G)\}\subseteq \operatorname{Lin}_{\mathbb{C}}(\mathcal{O}(\G), \mathbb{C})$$
which becomes a $^*$-algebra for
$$(\omega \chi)(x) := (\omega \odot \chi)\Delta(x), \quad \omega^*(x) = \overline{\omega(S(x)^*)}$$
where $\omega, \chi\in c_c(\tilde{\G})$ and $x\in \mathcal{O}(\G)$. We have a faithful $*$-representation
$$\rho_{\tilde{\G}}: c_c(\tilde{\G}) \to \ell^\infty(\check{\G})\subseteq B(L^2(\G)), \quad \rho_{\tilde{\G}}(\omega) \Lambda_\Phi(x) = \Lambda_\Phi((\id \odot \omega)\Delta(x)) = \Lambda_\Phi(x_{(1)}) \omega(x_{(2)}).$$
We use this $*$-representation to view $c_c(\tilde{\G})\subseteq \ell^\infty(\check{\G})$. 
We define the projections $p_\pi:= \rho_{\tilde{\G}}(\Phi(-\chi_\pi^*))$ and we note that $\sum_{\pi\in \operatorname{Irr}(\G)} p_\pi = 1$ with the sum converging strongly in $B(L^2(\G)).$

Let $(X, \alpha)$ be a $\G$-operator system. Given $\pi \in \operatorname{Irr}(\G)$, let $X_\pi$ be the associated spectral subspace and let $\mathcal{X}:= \bigoplus_{\pi\in \operatorname{Irr}(\G)}^{\operatorname{alg}} X_\pi$ be the associated algebraic core \cite[Section 3]{DH24}. We define the vector space
$$\mathcal{X}\rtimes_{\alpha, \operatorname{alg}}\G := \operatorname{span}\{\alpha(x)(1\otimes y): x \in \mathcal{X}, y \in c_c(\tilde{\G})\} \subseteq X  \odot B(L^2(\G)).$$

The following lemma is known, see e.g.\ the argument in \cite[Theorem 5.31]{DC17} in the $C^*$-context. We sketch the proof for the convenience of the reader.
\begin{Lem}
   Let $(M, \alpha)$ be a $\G$-dynamical von Neumann algebra with algebraic core $\mathcal{M}$. The vector space $(1\otimes p_\pi)(\mathcal{M}\rtimes_{\alpha,\operatorname{alg}} \G)(1\otimes p_\pi')$ is $\sigma$-weakly closed in $M\rtimes_\alpha \G$ for all $\pi, \pi'\in \Irr(\G)$.
\end{Lem}
\begin{proof} As in the proof of \cite[Theorem 5.31]{DC17}, there is a finite subset $F\subseteq \Irr(\G)$ such that 
$$(1\otimes p_\pi)(\mathcal{M}\rtimes_{\alpha,\operatorname{alg}} \G)(1\otimes p_{\pi'}) \subseteq \bigoplus_{\pi'' \in F} M_{\pi''} \odot (p_\pi B(L^2(\G)) p_{\pi'}).$$
Since $\bigoplus_{\pi'' \in F} M_{\pi''}$ is $\sigma$-weakly closed in $M$, we deduce that 
$$\bigoplus_{\pi''\in F} M_{\pi''}\odot (p_\pi B(L^2(\G)) p_{\pi'})$$
is $\sigma$-weakly closed in $M\bar{\otimes} p_\pi B(L^2(\G)) p_{\pi'}$. So, if $z_i \in (1\otimes p_\pi)(\mathcal{M}\rtimes_{\alpha,\operatorname{alg}}\G)(1\otimes p_{\pi'})$ is a net that converges $\sigma$-weakly to $z \in M \ovot B(L^2(\G))$, then we deduce that 
$$z \in \bigoplus_{\pi'' \in F} M_{\pi''}\odot (p_\pi B(L^2(\G)) p_{\pi'})\subseteq \mathcal{M}\odot p_\pi B(L^2(\G)) p_{\pi'}.$$

Next, define a unitary $U \in B(L^2(\G)\otimes L^2(\G))$ by
$$U(\Lambda_\Phi(a)\otimes \Lambda_\Phi(b)) = (\Lambda_\Phi \odot \Lambda_\Phi)(\Delta^{\op}(a)(1\otimes b)) = \Lambda_\Phi(a_{(2)})\otimes \Lambda_\Phi(a_{(1)}b), \quad a,b \in \mathcal{O}(\G)$$
and define the normal $*$-homomorphism
$$\beta: M \ovot B(L^2(\G)) \to M \ovot B(L^2(\G))\ovot B(L^2(\G)): s \mapsto U_{23}^*\Sigma_{23}(\alpha \otimes \id)(s)\Sigma_{23}U_{23}.$$
By \cite[Lemma 5.30]{DC17}, we have
$$\{s \in \mathcal{M}\odot p_\pi B(L^2(\G)) p_{\pi'}: \beta(s) = s\otimes 1\} = (1\otimes p_\pi)(\mathcal{M}\rtimes_{\alpha,\operatorname{alg}} \G)(1\otimes p_{\pi'}).$$
Therefore, we find that 
$z \in (1\otimes p_\pi)(\mathcal{M}\rtimes_{\alpha,\operatorname{alg}} \G)(1\otimes p_{\pi'})$.
\end{proof}

\begin{Lem}
    Let $(X, \alpha)$ be a $\G$-operator system and $z \in X \rtimes_\alpha \G$. Then for $\pi, \pi'\in \Irr(\G)$, we have
    $$(1\otimes p_\pi)z(1\otimes p_{\pi'}) \in \mathcal{X}\rtimes_{\alpha,\operatorname{alg}}\G.$$
\end{Lem}
\begin{proof} 
Let $(M, \beta)$ be a $\G$-dynamical von Neumann algebra such that $(X, \alpha)\subseteq_\G (M, \beta).$ Since $\mathcal{M}\rtimes_{\beta, \operatorname{alg}}\G$ is $\sigma$-weakly dense in $M\rtimes_\beta \G$,  we see that $(1\otimes p_\pi)(\mathcal{M}\rtimes_{\beta,\operatorname{alg}} \G)(1\otimes p_{\pi'})$ is $\sigma$-weakly dense in $(1\otimes p_\pi)(M\rtimes_\beta\G)(1\otimes p_{\pi'})$. By the preceding lemma, we find
$$ (1\otimes p_\pi)(M\rtimes_\beta \G)(1\otimes p_{\pi'})=(1\otimes p_\pi)(\mathcal{M}\rtimes_{\beta,\operatorname{alg}} \G)(1\otimes p_{\pi'})\subseteq \mathcal{M}\rtimes_{\beta,\operatorname{alg}} \G.$$

Thus, if $z\in X\rtimes_\alpha \G$, we see that 
$$v:=(1\otimes p_\pi)z(1\otimes p_{\pi'}) \in (\mathcal{M}\rtimes_{\beta,\operatorname{alg}} \G)\cap (X \ovot B(L^2(\G))).$$
We can therefore write 
$$v= \sum_{i=1}^n \beta(m_i)(1\otimes y_i)$$
with $m_i\in \mathcal{M}$ and $y_i\in c_c(\tilde{\G})$ where $\{y_i\}_{i=1}^n$ are linearly independent. We can further write
$$\beta(m_i) = \sum_{\pi\in F}\sum_{j,k=1}^{n_\pi} (m_i)_{jk}^\pi\otimes U_\pi(e_j^\pi, e_k^\pi)$$
for a finite subset $F\subseteq \operatorname{Irr}(\G)$ (not depending on $i$)
and then we have
$$v= \sum_{i=1}^n \sum_{\pi\in F}\sum_{j,k=1}^{n_\pi} (m_i)_{jk}^\pi\otimes U_\pi(e_j^\pi, e_k^\pi) y_i.$$
It is well-known that the multiplication map
$$\mathcal{O}(\G) \odot c_c(\tilde{\G}) \to B(L^2(\G))$$
is injective, so we conclude that $(m_i)_{jk}^\pi \in X$ for all $i=1, \dots, n$, $\pi\in F$ and all $j,k=1, \dots, n_\pi$ (by the Hahn-Banach theorem and the definition of the Fubini tensor product). Therefore,
$m_i = (\id \odot \epsilon)\beta(m_i) \in X$, so that $v \in \mathcal{X}\rtimes_{\alpha, \operatorname{alg}}\G.$
\end{proof}

\begin{Lem}\label{rigidlemma2} Let $(X, \alpha)$ a $\G$-operator system. If $T: X\rtimes_\alpha\G \to X\rtimes_\alpha \G$ is a $\check{\G}$-equivariant unital completely positive map such that 
$T(\alpha(x)(1\otimes y)) = \alpha(x)(1\otimes y)$
for all $x\in X$ and $y\in L^\infty(\check{\G})$, then $T= \id_{X\rtimes_\alpha \G}$.
\end{Lem}
\begin{proof}
    Let $z\in X\rtimes_\alpha \G$. If $\pi, \pi' \in \Irr(\G)$, the element $(1\otimes p_\pi)z(1\otimes p_{\pi'}) \in \mathcal{X}\rtimes_{\alg} \G$, so we have by the assumption and a multiplicative domain argument that
    $$(1\otimes p_\pi)z(1\otimes p_{\pi'})=T((1\otimes p_\pi)z(1\otimes p_{\pi'}))=(1\otimes p_\pi) T(z) (1\otimes p_{\pi'}).$$
    Since the partial sums $\sum_{\pi\in F} p_\pi$ where $F$ ranges over the finite subsets of $\operatorname{Irr}(\G)$ converge strongly to the identity on $B(L^2(\G))$, we conclude that $T(z) = z$.
\end{proof}

Exactly as in the case for discrete quantum groups, we then obtain the following results also for compact quantum groups:

\begin{Prop}\label{rigid2}
    The following statements are equivalent for $\G$-operator systems $(X, \alpha)$ and $(Y, \beta)$:
    \begin{enumerate}
        \item $(Y, \iota: X \to Y)$ is a $\G$-rigid extension of $X$.
        \item $(Y\rtimes_\beta \G, \iota \rtimes \G: X\rtimes_\alpha \G \to Y\rtimes_\beta \G)$ is a $\check{\G}$-rigid extension of $X\rtimes_\alpha \G$.
    \end{enumerate}
\end{Prop}

\begin{Prop}\label{injective envelope crossed product2} Let $(X, \alpha)$ be a $\G$-operator system. Then 
    $I_{\check{\G}}^1(X\rtimes_\alpha \G)= I_\G^1(X)\rtimes \G$
    as $\check{\G}$-operator systems.
\end{Prop}

\section{Equivariant injectivity and inner amenability pass to closed quantum subgroups}

Given a locally compact quantum group $\G$, we have $L^2(\G)= L^2(\check{\G})$ and thus Tomita-Takesaki theory gives us a canonical anti-unitary modular operator 
$$\check{J}= \check{J}_\G: L^2(\G)\to L^2(\G).$$

Assume now that $\mathbb{H}$ is a \emph{Vaes-closed quantum subgroup} of $\G$ \cite{DKSS12}, i.e. $\mathbb{H}$ is a locally compact quantum group and there exists a unital, normal, faithful $*$-homomorphism $\check{\gamma}: L^\infty(\check{\mathbb{H}})\to L^\infty(\check{\G})$ such that $(\check{\gamma}\otimes \check{\gamma})\circ \check{\Delta}^{\mathbb{H}}= \check{\Delta}^{\mathbb{G}}\circ \check{\gamma}.$ We also define the faithful normal $*$-homomorphism
$$\hat{\gamma}: L^\infty(\hat{\mathbb{H}})\to L^\infty(\hat{\G}): x \mapsto \check{J}_\G\check{\gamma}(\check{J}_{\mathbb{H}}x^* \check{J}_{\mathbb{H}})^*\check{J}_\G$$
satisfying $(\hat{\gamma}\otimes \hat{\gamma})\circ \hat{\Delta}^{\mathbb{H}} = \hat{\Delta}^\G \circ \hat{\gamma}.$ Define the unitaries
$$V_{\G, \mathbb{H}}:= (\check{\gamma} \otimes \id)(V_\mathbb{H})\in L^\infty(\check{\G})\bar{\otimes} L^\infty(\mathbb{H}), \quad W_{\mathbb{H}, \G}:=(\id \otimes \hat{\gamma})(W_{\mathbb{H}}) \in L^\infty(\mathbb{H})\ovot L^\infty(\hat{\G})$$
and the normal $*$-homomorphisms
\begin{align*}
    &\Delta_r^{\G, \mathbb{H}}: B(L^2(\G))\to B(L^2(\G))\bar{\otimes}L^\infty(\mathbb{H}): x \mapsto V_{\G, \mathbb{H}}(x\otimes 1)V_{\G, \mathbb{H}}^*\\
    &\Delta_l^{\mathbb{H}, \G}: B(L^2(\G))\to L^\infty(\mathbb{H})\ovot B(L^2(\G)): x \mapsto W_{\mathbb{H}, \G}^*(1\otimes x)W_{\mathbb{H}, \G}.
\end{align*}

 The map $\Delta_r^{\G, \mathbb{H}}$ defines an action $B(L^2(\G))\curvearrowleft \mathbb{H}$ and the map $\Delta_l^{\mathbb{H}, \G}$ defines an action $\mathbb{H}\curvearrowright B(L^2(\G))$. These maps restrict to coactions
 $$\Delta^{\G, \mathbb{H}}: L^\infty(\G)\to L^\infty(\G)\ovot L^\infty(\mathbb{H}), \quad \Delta^{\mathbb{H}, \G}: L^\infty(\G)\to L^\infty(\mathbb{H}) \ovot L^\infty(\G).$$

 In what follows, we will make use of the following list of well-known identities:
 \begin{align}
    \label{a} (\Delta_r^\G \otimes \id)\circ \Delta_r^{\G, \mathbb{H}}&= (\id \otimes \Delta_r^{\G, \mathbb{H}})\circ \Delta_r^\G,\\
   \label{b} (\id \otimes \Delta_l^\G)\circ \Delta_l^{\mathbb{H}, \G} &=(\Delta_l^{\mathbb{H}, \G} \otimes \id)\circ \Delta_l^\G,\\
   \label{c} (\id \otimes \Delta^{\mathbb{H}, \G})\circ \Delta^\G &= (\Delta^{\G, \mathbb{H}}\otimes \id)\circ \Delta^\G.
\end{align}
 
 If $(M, \alpha)$ is a $\G$-dynamical von Neumann algebra, there exists a unique coaction $\alpha_{\mathbb{H}}: M\to M\ovot L^\infty(\mathbb{H})$ such that 
   \begin{equation}\label{d}
       (\id \otimes \Delta^{\G, \mathbb{H}})\circ \alpha = (\alpha\otimes \id)\circ \alpha_{\mathbb{H}}.
   \end{equation}
   In this way, we obtain the $\mathbb{H}$-dynamical von Neumann algebra $(M, \alpha_{\mathbb{H}})$. We call $\alpha_{\mathbb{H}}$ the \emph{restriction of the action} $\alpha$ to the quantum subgroup $\mathbb{H}$. We also have the identity 
   \begin{equation}\label{e}
       (\id \otimes \Delta^{\mathbb{H}, \G})\circ \alpha= (\alpha_{\mathbb{H}}\otimes \id) \circ \alpha
   \end{equation}
   since we have (making use of \eqref{d} and \eqref{c})
   \begin{align*}
       (\alpha \otimes \id \otimes \id)(\alpha_{\mathbb{H}}\otimes \id)\alpha &= (\id \otimes \Delta^{\G, \mathbb{H}}\otimes \id) (\alpha \otimes \id)\alpha\\
       &= (\id \otimes \Delta^{\G, \mathbb{H}}\otimes \id) (\id \otimes \Delta^\G)\alpha\\
       &= (\id \otimes \id \otimes \Delta^{\mathbb{H}, \G})(\id \otimes \Delta^\G)\alpha\\
       &= (\id \otimes \id \otimes \Delta^{\mathbb{H}, \G})(\alpha \otimes \id)\alpha= (\alpha \otimes \id \otimes \id)(\id \otimes \Delta^{\mathbb{H}, \G})\alpha.
   \end{align*}

   We will need the following non-trivial result:
\begin{Prop}\label{crossed}
    There is a unique isometric normal $*$-homomorphism
    $\iota: M\rtimes_{\alpha_{\mathbb{H}}}\mathbb{H} \to M\rtimes_\alpha \G$
    such that 
    $$\iota(1\otimes y) = 1\otimes \check{\gamma}(y), \quad \iota(\alpha_{\mathbb{H}}(m)) = \alpha(m), \qquad m \in M, y \in L^\infty(\check{\mathbb{H}}).$$
\begin{itemize}
    \item Consider the $\check{\G}$-action 
    $$\beta:= (\id \otimes \id \otimes \check{\gamma}) \circ (\id \otimes \check{\Delta}_r^{\mathbb{H}}): M\rtimes_{\alpha_{\mathbb{H}}} \mathbb{H}\to (M\rtimes_{\alpha_{\mathbb{H}}} \mathbb{H})\ovot L^\infty(\check{\G}).$$
    Then
    $\iota: (M\rtimes_{\alpha_{\mathbb{H}}} \mathbb{H}, \beta)\to (M\rtimes_\alpha \G, \id \otimes \check{\Delta}_r^{\G})$
    is $\check{\G}$-equivariant.
    \item Consider the $\mathbb{H}$-action 
    $$\id \otimes \Delta_r^{\G, \mathbb{H}}: M\rtimes_\alpha \G\to (M\rtimes_\alpha \G)\ovot L^\infty(\mathbb{H}).$$
    Then $\iota: (M\rtimes_{\alpha_{\mathbb{H}}} \mathbb{H}, \id \otimes \Delta_r^{\mathbb{H}})\to (M\rtimes_\alpha \G, \id \otimes \Delta_r^{\G, \mathbb{H}})$ is $\mathbb{H}$-equivariant.
\end{itemize}
\end{Prop}
\begin{proof}
Assume first that $M$ is represented on a Hilbert space $\mathcal{H}$ in standard position. Let $U_\alpha \in B(\mathcal{H})\ovot L^\infty(\G)$ be the unitary implementation of $\alpha$ and let $U_{\alpha_{\mathbb{H}}}$ be the unitary implementation of $\alpha_{\mathbb{H}}$. In \cite[Section 6.5]{DC09} (see in particular equation $(6.2)$ in \cite{DC09}; note however that we need to make the switch from left to right coactions), it is shown that 
$$V_{\G, \mathbb{H},23}U_{\alpha,12}V_{\G, \mathbb{H},23}^*=(\id \otimes \Delta_{\G, \mathbb{H}})(U_\alpha)= U_{\alpha,12}U_{\alpha_{\mathbb{H}},13}$$
or equivalently
$$U_{\alpha_{\mathbb{H}},13}V_{\G, \mathbb{H},23} = U_{\alpha,12}^* V_{\G,\mathbb{H},23} U_{\alpha,12}.$$
Note that by Lemma \ref{well-defined}, we know that
$$U_{\alpha_{\mathbb{H}}}^*(M\rtimes_{\alpha_{\mathbb{H}}} \mathbb{H})U_{\alpha_{\mathbb{H}}} \subseteq B(\mathcal{H})\ovot L^\infty(\check{\mathbb{H}})$$
so it makes sense to define
$$\iota: M \rtimes_{\alpha_{\mathbb{H}}} \mathbb{H}\to B(\mathcal{H}\otimes L^2(\G)): z \mapsto U_\alpha(\id \otimes \check{\gamma})(U_{\alpha_{\mathbb{H}}}^* z U_{\alpha_{\mathbb{H}}})U_\alpha^*.$$

Note that 
$$V_{\mathbb{H},23} U_{\alpha_{\mathbb{H}},12} V_{\mathbb{H},23}^*=(\id \otimes \Delta_{\mathbb{H}})(U_{\alpha_{\mathbb{H}}}) = U_{\alpha_{\mathbb{H}},12} U_{\alpha_{\mathbb{H}},13}$$
so that 
$$U_{\alpha_{\mathbb{H}},12}^* V_{\mathbb{H},23} U_{\alpha_{\mathbb{H}},12}= U_{\alpha_{\mathbb{H}},13} V_{\mathbb{H}, 23}.$$

Therefore, if $\omega \in B(L^2(\mathbb{H}))_*$, we have
\begin{align*} (\id \otimes \check{\gamma})(U_{\alpha_{\mathbb{H}}}^*(1\otimes(\id \otimes \omega)(V_{\mathbb{H}}))U_{\alpha_{\mathbb{H}}})&= (\id \otimes \check{\gamma})(\id \otimes \id \otimes \omega)(U_{\alpha_{\mathbb{H}},12}^* V_{\mathbb{H}, 23} U_{\alpha_{\mathbb{H}},12})\\
&= (\id \otimes \check{\gamma})(\id \otimes \id \otimes \omega)(U_{\alpha_{\mathbb{H}},13}V_{\mathbb{H},23})\\
&=(\id \otimes \id \otimes \omega)(U_{\alpha_{\mathbb{H}},13}V_{\G,\mathbb{H},23})\\
&= (\id \otimes \id \otimes \omega)(U_{\alpha,12}^* V_{\G,\mathbb{H},23} U_{\alpha,12})\\
&= U_{\alpha}^*(1\otimes (\id \otimes \omega)(V_{\mathbb{\G, \mathbb{H}}}))U_\alpha\\
&= U_\alpha^*(1\otimes \check{\gamma}((\id \otimes \omega)(V_{\mathbb{H}})))U_\alpha
\end{align*}
so that 
$$(\id \otimes \check{\gamma})(U_{\alpha_{\mathbb{H}}}^*(1\otimes x))U_{\alpha_{\mathbb{H}}}) = U_\alpha^*(1\otimes \check{\gamma}(x))U_\alpha, \quad x \in L^\infty(\check{\mathbb{H}}).$$
Hence, we find for $x \in L^\infty(\check{\mathbb{H}})$ that \begin{align*}
    \iota(1\otimes x) =  U_\alpha(\id \otimes \check{\gamma})(U_{\alpha_{\mathbb{H}}}^* (1\otimes x)U_{\alpha_{\mathbb{H}}})U_\alpha^* = U_\alpha U_\alpha^*(1\otimes \check{\gamma}(x))U_\alpha U_\alpha^* = 1\otimes \check{\gamma}(x)
\end{align*}
and if $x\in M$, we have
$\iota(\alpha_{\mathbb{H}}(x)) = U_\alpha(x\otimes 1)U_\alpha^*= \alpha(x)$
and thus a map $\iota: M \rtimes_{\alpha_{\mathbb{H}}} \mathbb{H}\to M\rtimes_\alpha \G$ as in the proposition exists.

In the general case, choose a $*$-isomorphism $\pi: M \to N$ where $N$ is standardly represented on a Hilbert space $\mathcal{H}$. Define the coaction
$$\beta:= (\pi\otimes \id)\circ\alpha\circ\pi^{-1}: N \to N \ovot L^\infty(\G).$$
Then note that also
$$\beta_{\mathbb{H}}= (\pi\otimes \id)\circ \alpha_{\mathbb{H}}\circ \pi^{-1}: N \to N\ovot L^\infty(\mathbb{H})$$
since we have
\begin{align*}(\beta\otimes \id)\circ \beta_{\mathbb{H}}=(\id \otimes \Delta^{\G, \mathbb{H}})\circ \beta &= (\id \otimes \Delta^{\G, \mathbb{H}})\circ (\pi\otimes \id)\circ \alpha \circ \pi^{-1}\\
&=(\pi\otimes \id \otimes \id)\circ (\id \otimes \Delta^{\G, \mathbb{H}}) \circ \alpha \circ \pi^{-1}\\
&=(\pi\otimes \id \otimes \id)\circ (\alpha\otimes \id)\circ \alpha_{\mathbb{H}}\circ \pi^{-1}\\
&= (\beta \otimes \id) \circ (\pi\otimes \id) \circ \alpha_{\mathbb{H}}\circ \pi^{-1}.\end{align*}

Therefore, the composition
$$M\rtimes_{\alpha_{\mathbb{H}}} \mathbb{H} \cong N \rtimes_{\beta_{\mathbb{H}}} \mathbb{H}\hookrightarrow N\rtimes_\beta \G \cong M \rtimes_\alpha \G$$
is the desired map. Checking the $\check{\G}$ and $\mathbb{H}$-equivariance of $\iota$ is straightforwardly verified on the canonical generators of the crossed product.
\end{proof}

 Given a $\mathbb{H}$-dynamical von Neumann algebra $(N, \beta)$, we define the cotensor product 
   $$N   \overset{\mathbb{H}}{\square} L^\infty(\G):= \{z \in N \ovot L^\infty(\G): (\beta \otimes \id)(z) = (\id \otimes \Delta^{\mathbb{H}, \G})(z)\}.$$

By a routine argument involving slice maps and making use of equation \eqref{b}, we see that $$(\id \otimes \Delta^\G)(N \overset{\mathbb{H}}{\square} L^\infty(\G))\subseteq (N \overset{\mathbb{H}}{\square} L^\infty(\G))\ovot L^\infty(\G)$$ so that $N \overset{\mathbb{H}}{\square} L^\infty(\G)$ becomes a $\G$-dynamical von Neumann algebra. In particular, starting from a $\G$-dynamical von Neumann algebra $(M, \alpha)$, we have the associated $\mathbb{H}$-dynamical von Neumann algebra $(M, \alpha_{\mathbb{H}})$ and we can form the $\G$-dynamical von Neumann algebra $$M \overset{\mathbb{H}}{\square} L^\infty(\G)= \{z \in M \ovot L^\infty(\G): (\alpha_{\mathbb{H}}\otimes \id)(z) = (\id \otimes \Delta^{\mathbb{H}, \G})(z)\}.$$

If $m\in M$, equation \eqref{e} shows that $\alpha(m) \in M \overset{\mathbb{H}}{\square} L^\infty(\G)$ and therefore the map
 $\alpha: M \to M \overset{\mathbb{H}}{\square} L^\infty(\G)$ is a $\G$-equivariant isometric $*$-homomorphism. 
 
 Note that 
$\mathbb{C} \overset{\mathbb{H}}{\square} L^\infty(\G)\cong \{x \in L^\infty(\G): \Delta^{\mathbb{H}, \G}(x) = 1 \otimes x\} = L^\infty(\mathbb{H}\backslash \mathbb{G})$.
In view of this, the following result and its proof are a generalisation of \cite[Theorem 3.2]{Cr17b}.

\begin{Theorem}\label{main application}
   Let $(M, \alpha)$ be a $\G$-dynamial von Neumann algebra. The following statements are equivalent:
   \begin{enumerate}
       \item $(M, \alpha)$ is $\G$-injective.
       \item $(M, \alpha_{\mathbb{H}})$ is $\mathbb{H}$-injective and there exists a $\G$-equivariant unital completely positive map $P: M \overset{\mathbb{H}}{\square} L^\infty(\G) \to M$ such that $P\circ \alpha = \id_M$.
   \end{enumerate}
\end{Theorem}

\begin{proof}$(1)\implies (2)$ Consider the $\check{\G}$-action
$$\beta= (\id \otimes \id \otimes \check{\gamma})\circ (\id \otimes \check{\Delta}_r^{\mathbb{H}}): M \ovot B(L^2(\mathbb{H}))\to (M \ovot B(L^2(\mathbb{H})))\ovot L^\infty(\check{\G}).$$

By $\check{\G}$-injectivity of $M\rtimes_\alpha\G$ and $\check{\G}$-equivariance of the canonical map $\iota$, there is a $\check{\G}$-equivariant unital completely positive map
$$E: (M\bar{\otimes} B(L^2(\mathbb{H})), \beta)\to (M\rtimes_\alpha \G, \id \otimes \check{\Delta}_r^\G).$$
extending $\iota$. Then, we have for $z\in M\bar{\otimes} B(L^2(\mathbb{H}))$ that
\begin{align*}
    (E\otimes \id)(\id \otimes \Delta_r^{\mathbb{H}})(z)&= (E\otimes \id)(V_{\mathbb{H},23}(z\otimes 1)V_{\mathbb{H},23}^*)= V_{\mathbb{\G, \mathbb{H}},23}(E(z)\otimes 1) V_{\mathbb{\G, \mathbb{H}},23}^*=(\id \otimes \Delta_r^{\G, \mathbb{H}})E(z) 
\end{align*}
where we used that $V_{\mathbb{H},23}$ is in the multiplicative domain of $E\otimes \id$ and that $E(1\otimes y) = 1\otimes \check{\gamma}(y)$ for $y\in L^\infty(\check{\mathbb{H}})$. Further, if $z\in M \bar{\otimes}L^\infty(\mathbb{H})$, we find
\begin{align*}
    (\id \otimes \check{\Delta}_r^{\mathbb{G}})E(z)&= (E\otimes \id)\beta(z)= (E\otimes \id)(z\otimes 1) = E(z)\otimes 1
\end{align*}
where we used that $\check{V}_{\mathbb{H}} \in L^\infty(\mathbb{H})'\bar{\otimes} L^\infty(\check{\mathbb{H}})$. Therefore, $E(z)\in \alpha(M)$. Thus, it makes sense to define
$$F:= \alpha^{-1}\circ E: M\bar{\otimes}L^\infty(\mathbb{H})\to M.$$
Note that if $a\in M$, we have
$$F(\alpha_{\mathbb{H}}(a))=\alpha^{-1}(\iota(\alpha_\mathbb{H}(a))) = \alpha^{-1}(\alpha(a)) = a.$$
Further, we check that $F$ is $\mathbb{H}$-equivariant:
\begin{align*}
    (\alpha\otimes \id)(F\otimes \id)(\id \otimes \Delta_{\mathbb{H}})&= (E\otimes \id)(\id \otimes \Delta_\mathbb{H})= (\id \otimes \Delta^{\G, \mathbb{H}})\circ E= (\id \otimes \Delta^{\G, \mathbb{H}})\circ \alpha\circ F= (\alpha\otimes \id)\circ \alpha_{\mathbb{H}} \circ F
\end{align*}
so by injectivity of $\alpha\otimes \id$ we conclude that
$(F\otimes \id)\circ (\id \otimes \Delta_{\mathbb{H}}) = \alpha_{\mathbb{H}}\circ F.$
Thus, $F$ is a $\mathbb{H}$-equivariant completely positive map satisfying $F\circ \alpha_{\mathbb{H}}= \id_M$, and thus $\alpha_{\mathbb{H}}:M\curvearrowleft\mathbb{H}$ is an amenable action. We conclude that $(M, \alpha_\mathbb{H})$ is $\mathbb{H}$-injective.

$(2)\implies (1)$ Let $E_{\mathbb{H}}: (M \ovot L^\infty(\mathbb{H}), \id \otimes \Delta^{\mathbb{H}}) \to (M, \alpha_{\mathbb{H}})$ be a unital completely positive map such that $E_{\mathbb{H}}\circ \alpha_{\mathbb{H}}= \id_M$. Given $z \in M \ovot L^\infty(\G)$, we have
\begin{align*}
    (\alpha_{\mathbb{H}}\otimes \id)(E_{\mathbb{H}}\otimes \id)(\id \otimes \Delta^{\mathbb{H}, \G})(z)&= (E_{\mathbb{H}} \otimes \id \otimes \id)(\id \otimes \Delta^{\mathbb{H}}\otimes \id)(\id \otimes \Delta^{\mathbb{H}, \G})(z)\\
    &= (E_{\mathbb{H}}\otimes \id \otimes \id)(\id\otimes \id \otimes \Delta^{\mathbb{H}, \G})(\id \otimes \Delta^{\mathbb{H}, \G})(z)\\
    &=(\id \otimes \Delta^{\mathbb{H}, \G})(E_{\mathbb{H}}\otimes \id)(\id \otimes \Delta^{\mathbb{H}, \G})(z).
\end{align*}
Therefore, we can define the map
$$Q: M \ovot L^\infty(\G) \to M \overset{\mathbb{H}}{\square} L^\infty(\G): z \mapsto (E_{\mathbb{H}}\otimes \id)(\id \otimes \Delta^{\mathbb{H}, \G})(z)$$
which is $\G$-equivariant since by \eqref{b} we have
\begin{align*}
    (Q \otimes \id)(\id \otimes \Delta^\G)(z)&= (E_{\mathbb{H}}\otimes \id \otimes \id)(\id \otimes \Delta^{\mathbb{H}, \G}\otimes \id)(\id \otimes \Delta^\G)(z)\\
    &= (E_{\mathbb{H}}\otimes \id\otimes \id)(\id \otimes \id \otimes \Delta^\G)(\id \otimes \Delta^{\mathbb{H}, \G})(z)=(\id \otimes \Delta^\G) Q(z).
    \end{align*}
 We define the $\G$-equivariant map
    $$E_\G:= P\circ Q: M \ovot L^\infty(\G) \to M.$$
    If $m \in M$, we have 
    $$E_\G(\alpha(m)) = P(Q(\alpha(m)))= P(E_{\mathbb{H}}\otimes \id)(\alpha_{\mathbb{H}}\otimes \id)\alpha(m)=  P(\alpha(m)) = m$$
    so the map $E_\G$ witnesses the amenability of the action $\alpha: M \curvearrowleft \G$.
\end{proof}

  \begin{Rem}
      It is in general not clear that if $(M, \alpha)$ is $\mathbb{G}$-amenable and $\mathbb{H}$ is a closed quantum subgroup of $\G$, then $(M, \alpha_{\mathbb{H}})$ is $\mathbb{H}$-amenable. However, Theorem \ref{main application} shows that this is true if the von Neumann algebra $M$ is injective. For classical locally compact groups, it was established in \cite[Proposition 2.6]{AD82} that amenability of actions passes to closed subgroups, but under an additional countability assumption on the locally compact group. In the case that the von Neumann algebra involved is injective, we can therefore dispose of these countability assumptions. 
  \end{Rem}  

    The situation for inner amenability of actions is much clearer, as the following proposition shows.

\begin{Prop}\label{subgroups}
Let $(M, \alpha)$ be a $\G$-dynamical von Neumann algebra and $\mathbb{H}$ a Vaes-closed quantum subgroup of $\G$.  If $(M, \alpha)$ is $\G$-inner amenable, then $(M, \alpha_{\mathbb{H}})$ is inner $\mathbb{H}$-amenable.
\end{Prop}
\begin{proof}
   Choose a $\G$-equivariant unital completely positive map
    $$E_\G: (M\rtimes_\alpha \G, \id \otimes \Delta_r^\G)\to (M, \alpha), \quad E_{\mathbb{G}}\circ \alpha = \id_M.$$
    Define then the unital completely positive map 
    $$E_{\mathbb{H}}:= E_\G \circ \iota: M\rtimes_{\alpha_{\mathbb{H}}}\mathbb{H}\to M.$$
If $z\in M\rtimes_{\alpha_{\mathbb{H}}} \mathbb{H}$, we find (using equations \eqref{a}, \eqref{d} and $\mathbb{H}$-equivariance of $\iota$) that
\begin{align*}
    (\alpha \otimes \id)\alpha_{\mathbb{H}}(E_{\mathbb{H}}(z))&= (\id \otimes \Delta^{\mathbb{G}, \mathbb{H}})\alpha(E_\G(\iota(z)))\\
    &= (\id \otimes \Delta^{\G, \mathbb{H}})(E_\G \otimes \id)(\id \otimes \Delta_r^\G)\iota(z)\\
    &= (E_\G \otimes \id \otimes \id)(\id \otimes \id \otimes \Delta^{\G, \mathbb{H}})(\id \otimes \Delta_r^{\G})\iota(z)\\
    &= (E_\G\otimes \id \otimes \id)(\id \otimes \Delta_r^\G\otimes \id)(\id \otimes \Delta_r^{\G, \mathbb{H}})\iota(z)\\
    &=(E_\G \otimes \id \otimes \id)(\id \otimes \Delta_r^\G \otimes \id)(\iota \otimes \id)(\id \otimes \Delta_r^{\mathbb{H}})(z)\\
    &=(\alpha \otimes \id)(E_{\mathbb{H}}\otimes \id)(\id \otimes \Delta_r^{\mathbb{H}})(z)
\end{align*}
so $E_{\mathbb{H}}: (M\rtimes_{\alpha_{\mathbb{H}}}\mathbb{H}, \id \otimes \Delta_r^{\mathbb{H}})\to (M, \alpha_{\mathbb{H}})$ is $\mathbb{H}$-equivariant. Trivially, $E_{\mathbb{H}}\circ \alpha_{\mathbb{H}}= \id_M$ so $(M, \alpha_{\mathbb{H}})$ is inner $\mathbb{H}$-amenable.\end{proof}

\begin{Rem}
    If we apply Proposition \ref{subgroups} to the trivial action $\mathbb{C}\curvearrowleft\G$, we recover the fact that inner amenability of a locally compact quantum group passes to Vaes-closed quantum subgroups \cite[Proposition 3.10]{Cr19}. 
\end{Rem}

\textbf{Acknowledgments:} The research of the author was supported by Fonds voor Wetenschappelijk Onderzoek (Flanders), via an FWO Aspirant fellowship, grant 1162522N. The author would like to thank Kenny De Commer for valuable discussions throughout this entire project, as well as Benjamin Anderson-Sackaney for pointing out to us the mistake in the proof of \cite[Theorem 4.7]{Moa18} and Jacek Krajczok for a useful discussion on the approximation property of a locally compact quantum group $\G$. The author would also like to express gratitude to Stefaan Vaes for noting that an additional co-amenability assumption in Corollary \ref{coro} could be removed, and to the anonymous referee for their useful comments.

\end{document}